\documentclass[11pt]{article}

\usepackage{amstext,amssymb,amsmath,amsbsy}
\usepackage{hyperref}
\usepackage{amscd}
\usepackage{amsfonts}
\usepackage{indentfirst}
\usepackage{verbatim}
\usepackage{amsmath}
\usepackage{amsthm}
\usepackage{enumerate}
\usepackage{graphicx}
\usepackage{color}
\usepackage[OT1]{fontenc}
\usepackage[latin1]{inputenc}
\usepackage[english]{babel}
\usepackage{amssymb}
\usepackage{subfig}
\usepackage{algorithm}
\usepackage{algpseudocode}

\newcommand{\R}{\mathbb{R}}
\newcommand{\ds}{\displaystyle}

\newcommand{\x}{{\bf x}}
\newcommand{\Div}{{\rm div}}

\newtheorem{Theorem}{Theorem}[section]
\newtheorem{Lemma}{Lemma}[section]

\newtheorem{Proposition}{Proposition}[section]
\newtheorem{Corollary}{Corollary}[section]
\newtheorem{remark}{Remark}[section]

\newtheorem*{Assumption*}{Assumption}

\newtheorem{problem}{Problem}[section]
\newtheorem*{problem*}{Problem}
\numberwithin{equation}{section}

\begin{document}

\title{A convergent numerical method to recover the initial condition of nonlinear parabolic equations from lateral Cauchy data\footnote{Dedicated to the 70th anniversary of a distinguished expert in the field of inverse problems Professor Michael V. Klibanov}}

\author{Thuy T. Le\thanks{%
Department of Mathematics and Statistics, University of North Carolina at
Charlotte, Charlotte, NC 28223, USA, \text{tle55@uncc.edu}.} \and Loc H. Nguyen 
\thanks{%
Department of Mathematics and Statistics, University of North Carolina at
Charlotte, Charlotte, NC 28223, USA, \text{loc.nguyen@uncc.edu} (corresponding author).}}

\date{}
\maketitle
\begin{abstract}
	We propose a new numerical method for the solution of the problem of the reconstruction of the initial condition of a quasilinear parabolic equation from the measurements of both Dirichlet and Neumann data on the boundary of a bounded domain. Although this problem is highly nonlinear, we do not require an initial guess of the true solution. The key in our method is the derivation of a boundary value problem for a system of coupled quasilinear elliptic equations whose solution is the vector function of the spatially dependent Fourier coefficients of the solution to the governing parabolic equation. We solve this problem by an iterative method. The global convergence of the system is rigorously established using a Carleman estimate. Numerical examples are presented.
\end{abstract}



\section{Introduction}

Let $d \geq 1$ be the spatial dimension and $T > 0$.
Let $q: \R \to \R$ and $c: \R^d \to \R$ be two smooth functions in the class $C^1$. 
Assume that $c(\x) \geq c_0$ for some $c_0 > 0.$
Consider the problem
\begin{equation}
	\left\{
		\begin{array}{rcll}
			c(\x)u_t(\x, t) &=& \Delta u(\x, t) + q(u(\x, t)) &\x \in \R^d, t \in (0, T)\\
			u(\x,0) &=& p(\x) & \x \in \R^d,
	\end{array}
	\right.
	\label{main eqn}
\end{equation}
where $p$ is a source function compactly supported in an open and bounded domain $\Omega$ of $\R^d$ with smooth boundary $\partial \Omega.$ 
We briefly discuss the unique solvability and some regularity properties of \eqref{main eqn}.
Assume that the initial condition of $p$ is in $H^{2 + \beta}(\R^d)$ for some $\beta \in [0, 1 + 4/d]$ and has compact support.
Assume further that 
\begin{equation}
	|q(s)| \leq C(1 + |s|) \quad \mbox{for all } s \in \R
	\label{Lady}
\end{equation} for some constant $C > 0$.
Then \eqref{main eqn} has a unique solution with $|u(\x, t)| \leq M$ and
 $u \in H^{2 + \beta, 1 + \beta/2} (\R^d \times [0, T])$ for some constant $M > 0$. 
 These unique solvability and regularity properties can be obtained by applying Theorem 6.1 in \cite[Chapter 5, \S 6]{LadyZhenskaya:ams1968} and Theorem 2.1 in \cite[Chapter 5, \S 2]{LadyZhenskaya:ams1968}.
 
We are interested in the following problem.
\begin{problem}[Inverse Source Problem]	
	Assume that there is a number $M > 0$ such that $|u(\x, t)| \leq M$ for all $\x \in \overline \Omega,$ $t \in [0, T]$. 	
	Given the lateral Cauchy data
	\begin{equation}
		f(\x, t) = u(\x, t) 
		\quad
		\mbox{and }
		\quad
		g(\x, t) = \partial_\nu u(\x, t)
		\label{data}
	\end{equation}
	for $\x \in \partial \Omega$, $t \in [0, T]$,
	determine the function $u(\x, 0) = p(\x), \x \in \Omega.$
	\label{ISP}
\end{problem}
	
	Problem \ref{ISP} arises from the problem of recovering the initial condition $p(\x)$ of parabolic equation \eqref{main eqn} from the lateral Cauchy data.
It has many real-world applications, e.g.,	 
determination of the spatially distributed temperature inside a solid from the boundary measurement of the heat and heat flux in the time domain \cite{Klibanov:ip2006}; 
identification the pollution on the surface of the rivers or lakes \cite{BadiaDuong:jiip2002};
 effective monitoring the heat conduction processes in steel industries, glass and polymer-forming and nuclear power station \cite{LiYamamotoZou:cpaa2009}.
When the nonlinear term $q(u)$ takes the form $u(1 - u)$ (or $q(u) = u(1 - |u|^{\alpha})$) for some $\alpha > 0$, the parabolic equation in (\ref{main eqn}) is called the high dimensional version of the well-known Fisher (or Fisher-Kolmogorov) equation \cite{Fisher:ae1937}.
Although the nonlinearity $q$ does not satisfy condition \eqref{Lady}, we do not experience any difficulty in numerical computations of the forward problem.
It is worth mentioning that the Fisher equation occurs in ecology, physiology, combustion, crystallization, plasma physics, and in general phase transition problems, see \cite{Fisher:ae1937}.
% Problem \ref{ISP} contributes greatly to those fields. 
%See the celebrated paper by Fisher \cite{Fisher:ae1937} for the derivation of the Fisher equation. 
Due to its realistic applications, the problem of determining the initial conditions of parabolic equations has been studied  intensively. However, up to the knowledge of the authors, numerical solutions are computed only in the case when the nonlinearity is absent, see e.g., \cite{LiNguyen:IPSE2019}. 
The uniqueness of Problem \ref{ISP} is well-known assuming that the nonlinearity $q$ is in class $C^1$, see \cite{Lavrentiev:AMS1986}. 
On the other hand, the logarithmic stability results were rigorously proved in \cite{Klibanov:ip2006, LiYamamotoZou:cpaa2009}. 
%Since the nonlinearities in our paper is allowed to grow faster than the ones imposed in \cite{Klibanov:ip2006, Lavrentiev:AMS1986, LiYamamotoZou:cpaa2009}, it is important to establish a uniqueness result.
For completeness, we briefly recall the logarithmic stability of Problem \ref{ISP} in this paper. 
The natural approach to solve this problem is the optimal control method; that means, minimizing some mismatch functionals.
However, since the known stability is logarithmic \cite{Klibanov:ip2006, LiYamamotoZou:cpaa2009}, the optimal control approach might not give good numerical results; especially, when the initial guess, if provided, is far away from the true solution. 
A more important reason for us to not use the optimal control method is that the cost functional is nonconvex and, therefore, might have multi-minima.
We draw the reader's attention to the convexification methods, see \cite{KlibanovIoussoupova:SMA1995, Klibanov:sjma1997, Klibanov:nw1997, KlibanovNik:ra2017, Klibanov:ip2015, KlibanovKolesov:cma2019, KlibanovLiZhang:ip2019, KhoaKlibanovLoc:arxiv2019}, which convexify the cost functional and therefore the difficulty about the lack of the initial guess is avoided. 
Applying the convexification method to numerically solve Problem \ref{ISP} will be studied in the near future project.
In this paper, rather than working on the convexification method, similarly to \cite{BAUDOUIN:SIAMNumAna:2017, Boulakia:preprint2019},  of which the authors have successfully solved a coefficient inverse problem for a hyperbolic equation and an inverse source problem for a parabolic equation by combining the contraction principle and a new Carleman estimate, we propose a numerical method for Problem \ref{ISP}.
The convergence of our method is proved based on the contraction principle using a new Carleman estimate. 
The latter is similar to the idea of \cite{BAUDOUIN:SIAMNumAna:2017 ,Baudouin:preprint2019, Boulakia:preprint2019}.

As mentioned, since a good initial guess of the true solution of Problem \ref{ISP} is not always available, the optimal control method, which is widely used in the scientific community, might not be applicable.
To overcome this difficulty, we propose to solve Problem \ref{ISP} in the Fourier domain. 
More precisely, we derive a system of elliptic PDEs whose solution consists of a finite number of the Fourier coefficients of the solution to the parabolic equation (\ref{main eqn}). 
The solution of this system directly yields the knowledge of the function $u(\x, t)$, from which the solution to our inverse problem follows. 
%Since the system of PDEs above is highly nonlinear, 
We numerically solve this nonlinear system by an iterative process. 
The initial solution can be computed by solving the system obtained by removing the nonlinear term.
Then, we approximate the nonlinear system by replacing the nonlinearity by the one acting on the initial solution obtained in the previous step. 
Solving this approximation system, we find an updated solution.
Continuing this process, we get a fast convergent sequence reaching to the desired function.
The convergence of this iterative procedure is rigorously proved by using a new Carleman estimate and the standard arguments of the contraction principle.
The fast convergence will be shown in both analytic and numerical senses.

Two papers closely related to the current one are \cite{Boulakia:preprint2019} and \cite{LiNguyen:IPSE2019}. In \cite{Boulakia:preprint2019}, a source term for a nonlinear parabolic equation is computed and  in \cite{LiNguyen:IPSE2019}, the second author and his collaborator computed the initial condition of the linear parabolic equation from the lateral Cauchy data.
On the other hand, the coefficient inverse problem for parabolic equations is also very interesting and studied intensively. 
We draw the reader's attention to \cite{Borceaetal:ip2014, CaoLesnic:nmpde2018, CaoLesnic:amm2019, KeungZou:ip1998, Nguyen:arxiv2019, Nguyens:jiip2019, YangYuDeng:amm2008} for important numerical methods and good numerical results.
Besides, the  problem of recovering the initial conditions for the hyperbolic equation is very interesting since it arises in many real-world applications. 
For instance, the problems thermo- and photo-acoustic tomography play the key roles in biomedical imaging. 
We refer the reader to some important works in this field \cite{LiuUhlmann:ip2015, KatsnelsonNguyen:aml2018, HaltmeierNguyen:SIAMJIS2017}. 
Applying the Fourier transform, one can reduce the problem of reconstructing the initial conditions for hyperbolic equations to some inverse source problems for the Helmholtz equation, see \cite{NguyenLiKlibanov:IPI2019, WangGuoZhangLiu:ip2017, WangGuoLiLiu:ip2017, LiLiuSun:IPI2018, ZhangGuoLiu:ip2018} for some recent results. 

The paper is organized as follows.
In Section \ref{sec method}, we derive a nonlinear system of elliptic PDEs, which leads to a numerical method to solve Problem \ref{ISP}. 
%In this section, we also briefly discuss the uniqueness of this inverse problem for the completeness' sake.
This nonlinear system is solved by an iterative scheme. The proof of the convergence of this iteration is based on the contraction principle.
%Especially, in Section \ref{sec solve nonlinear system}, 
In Section \ref{Sec Carleman}, we establish and prove a Carleman estimate. This estimate plays an important role in the proof Theorem \ref{minimizer} that guarantees the existence and uniqueness of the least-squares solution to over-determined elliptic systems.
In Section \ref{sec convergence}, we prove the convergence of the iterative sequence.
In Section \ref{sec num}, we discuss the implementation of our method and show several numerical results.
Section \ref{sec remarks} is for concluding remarks.

\section{A numerical method to solve Problem \ref{ISP} }	 \label{sec method}

The main aims of this section are to derive a system of nonlinear elliptic equations, whose solutions directly yield the solutions to Problem \ref{ISP}, and then propose a method to solve it.

\subsection{A system of nonlinear elliptic equations} \label{sec 2.1}

Let $\{\Psi_n\}_{n \geq 1}$ be an orthonormal basis of $L^2(0, T).$
For each point $\x \in \Omega$, we can approximate $u(\x, t)$, $t \in [0, T],$ as
\begin{equation}
	u(\x, t) = \sum_{n = 1}^\infty u_n(\x) \Psi_n(t) \simeq \sum_{n = 1}^N u_n(\x) \Psi_n(t)
	\label{Fourier u}
\end{equation} 
where
\begin{equation}
	u_n(\x) = \int_0^T u(\x, t) \Psi_n(t) dt \quad n \geq 1.
	\label{Fourier coefficient}
\end{equation}
%The choice of $\{\Psi_n\}_{n \geq 1}$ and the ``cut-off" number $N$ will be determined later.
%More precisely, in Section \ref{sec num}, we verify that with appropriate values of $N$, the error causing from \eqref{Fourier u} is small, see also Figure \ref{fig choose N}.
\begin{remark}
Replacing $\simeq$ in \eqref{Fourier u}  by $``="$   forms our approximate mathematical model. 
We cannot prove the convergence of the model as $N \to \infty$. Indeed, such a result is very hard to prove due to both the nonlinearity and the ill-posedness of our inverse problem. 
Therefore, our goal below is to find spatially dependent Fourier coefficients $u_n$ defined in \eqref{Fourier coefficient}. The number $N$ should be chosen numerically. 
In fact, in Section \ref{sec num}, we verify that with appropriate values of $N$, the error causing from \eqref{Fourier u} is small, see also Figure \ref{fig choose N}
\end{remark}
Due to (\ref{Fourier u}), the function $u_t(\x, t)$ is approximated by
\begin{equation}
	u_t(\x, t) \simeq \sum_{n = 1}^N u_n(\x) \Psi_n'(t)
	\quad \x \in \Omega, t \in [0, T].
	\label{Fourier ut}
\end{equation}
From now on, we replace the approximation ``$\simeq$" by equality. 
This obstacle will be considered numerically in Remark \ref{rem N} and Figure \ref{fig choose N}.
Plugging (\ref{Fourier u}) and (\ref{Fourier ut}) into the governing equation in (\ref{main eqn}), we obtain \begin{equation}
	c(\x)\sum_{n = 1}^N u_n(\x) \Psi_n'(t) = 
	\sum_{n = 1}^N \Delta u_n(\x) \Psi_n(t) 
	+ q\Big(\sum_{n = 1}^N u_n(\x) \Psi_n(t)\Big) \label{3.3}
\end{equation}
for all $\x \in \Omega.$
For each $m = 1, \dots, N$, multiply $\Psi_m(t)$ to both sides of (\ref{3.3}) and then integrate the resulting equation with respect to $t$ on $[0, T]$. 
For all $\x \in \Omega,$ we have
\begin{multline}
%	\hspace{-1cm}
	c(\x)\sum_{n = 1}^N u_n(\x) \int_0^T\Psi_n'(t) \Psi_m(t)dt 
	%\nonumber
	\\
	 = 
	\sum_{n = 1}^N \Delta u_n(\x) \int_0^T\Psi_n(t) \Psi_m(t)dt 
	+ \int_0^Tq\Big(\sum_{n = 1}^N u_n(\x) \Psi_n(t)\Big)\Psi_m(t)dt.
	\label{3.4}
\end{multline}
The system (\ref{3.4}) with $m = 1, \dots, N$ can be rewritten as
\begin{equation}
	c(\x) \sum_{n = 1}^N s_{mn} u_n(\x) = \Delta u_m(\x) + q_m(u_1(\x), u_2(\x), \dots, u_N(\x))
	\label{system U}
\end{equation}
where
\[	s_{mn} = \int_0^T \Psi'_n(t) \Psi_m(t)dt 
\]
and
\begin{equation}
	q_m(u_1(\x), u_2(\x), \dots, u_N(\x)) = \int_{0}^T q\Big(\sum_{n = 1}^N u_n(\x) \Psi_n(t)\Big) \Psi_m(t)dt.
	\label{eqnqm}
\end{equation}
%which is equivalent to
%\begin{equation}
%	c(\x) S U(\x) = \Delta U(\x) + Q(U(\x))
%	\label{system U}
%\end{equation}
%where $S$ is the $N \times N$ matrix with the $mn^{\rm th}$ entry 
%\[
%	 s_{mn} = \int_0^T \Psi_n'(t)\Psi_m(t) dt
%\]
%and 
%\[
%	Q(U(\x)) = (q_m(U))_{m = 1}^N = \left(\int_0^Tq\Big(\sum_{n = 1}^N u_n(\x) \Psi_n(t)\Big)\Psi_m(t)dt\right)_{m = 1}^N.
%\]
Due to (\ref{Fourier coefficient}), each function $u_m$, $m = 1, \dots, N$, satisfies the Cauchy boundary conditions
%\begin{equation}
%	u_m(\x) = f_m(\x) 
%	= \int_0^T f(\x, t) \Psi_m(t)dt 
%	\quad \mbox{and}
%	\quad
%	\partial_\nu u_m(\x) = g_m(\x)
%	=\int_0^T g(\x, t) \Psi_m(t)dt
%	\label{boundary conditions}
%\end{equation}
\begin{equation}
\left\{
	\begin{array}{rl}
		u_m(\x) &= f_m(\x) 
	= \ds\int_0^T f(\x, t) \Psi_m(t)dt 
\\
	\partial_\nu u_m(\x) &= g_m(\x)
	=\ds\int_0^T g(\x, t) \Psi_m(t)dt
	\end{array}
\right.
	\label{boundary conditions}
\end{equation}
for all $\x \in \partial \Omega$, $m = 1, \dots, N.$ Here, $f(\x, t)$ and $g(\x, t)$ are the given data.

\begin{remark}
Problem \ref{ISP} becomes the problem of finding all functions $u_m(\x)$, $\x \in \Omega$, $m = 1, \dots, N$, satisfying (\ref{system U}) and the Cauchy boundary conditions (\ref{boundary conditions}).
In fact, if all of those functions are known, we can compute the function $u(\x, t)$, $\x \in \Omega$, $t \in [0, T]$ via (\ref{Fourier u}). 
Then, the initial condition $p(\x)$ is given by the function $u(\x, 0).$
\end{remark}

\begin{remark}
	From now on, we consider the values of $f_m(\x)$ and $g_m(\x)$ on $\partial \Omega$, $m = 1, \dots, N$, as the ``indirect data", see (\ref{boundary conditions}). 
	Denote by $f_m^*(\x)$ and $g_m^*(\x)$ the noiseless data. In numerical study, we set the noisy data as
	\[
		f_m^{\delta} = f_m^*(1 + \delta(-1 + 2{\rm rand})) \quad 
		g_m^{\delta} = g_m^*(1 + \delta(-1 + 2{\rm rand}))
	\] on $\partial \Omega,$ $1 \leq m \leq N$ where $\delta > 0$ is the noise level and ${\rm rand}$ is the function taking uniformly distributed random numbers  in the range $[0, 1]$.
	In our numerical study, $\delta = 20\%.$
\end{remark}

\subsection{An iterative procedure to solve the system \eqref{system U}-- \eqref{boundary conditions}} \label{sec solve nonlinear system}

We propose a procedure to compute $u_1(\x)$, $ \dots,$ $u_N(\x)$.
We first approximate (\ref{system U})--(\ref{boundary conditions}) by solving the following over-determined problem
\begin{equation}
	\left\{
		\begin{array}{rcll}
			 c(\x) \sum_{n = 1}^N s_{mn} u_n^{(0)}(\x) &=& \Delta u_m^{(0)}(\x) 
 &\x \in \Omega,\\
			u_m^{(0)}(\x) &=& f_m(\x) &\x \in \partial \Omega,\\
			\partial_{\nu} u_m^{0}(\x) &=& g_m(\x) &\x \in \partial \Omega
		\end{array}
	\right.
	\quad m = 1, 2, \dots, N
	\label{U0}
\end{equation}
for a vector value function $(u^{(0)}_1, \dots, u^{(0)}_N)$. 
Then, assume by induction that we know $(u^{(k-1)}_1, \dots, u^{(k-1)}_N)$, $k \geq 1$, we find $(u^{(k)}_1, \dots, u^{(k)}_N)$ by solving
\begin{equation}
	\left\{
		\begin{array}{ll}
			c(\x) \sum_{n = 1}^N s_{mn} u^{(k)}_n(\x) = \Delta u^{(k)}_m(\x) 
			\\
			\hspace{3cm}+ q_m[P (u^{(k-1)}_1(\x)), \dots, P(u_N^{(k-1)}(\x))]
 &\x \in \Omega,\\
			u_m^{(k)}(\x) = f_m(\x) &\x \in \partial \Omega,\\
			\partial_{\nu} u_m^{(k)}(\x) = g_m(\x) &\x \in \partial \Omega
		\end{array}
	\right.
	\label{Up}
\end{equation}
where $q_m$ is defined in (\ref{eqnqm})
for $m = 1, 2, \dots, N.$ 
Here, 
\begin{equation}
	P(s) = \left\{
		\begin{array}{ll}
			M\sqrt{T} & s \in (M\sqrt{T}, \infty),\\
			s &   s \in [-M\sqrt{T}, M\sqrt{T}],\\
			-M\sqrt{T}& s \in (-\infty, -M\sqrt{T}]
		\end{array}
	\right.
	\quad \mbox{for all } s \in \R.
	\label{cutoff function}
\end{equation}
serves as a cut-off function.
where $M > \|u^*\|_{L^{\infty}(\Omega \times [0, T])}$ is a fixed constant.

%The implementation of this procedure, using the finite difference method, is described in Algorithm \ref{alg}.

%\begin{remark}
	In practice since both Dirichlet and Neumann conditions imposed, problem (\ref{U0}) and  problem (\ref{Up}) might have no solution. 
	However, since these two problems are linear, we can use the linear least-squares method to find the ``best fit" solutions. 
	In order to guarantee the convergence of the method, we include a Carleman weight function in the linear least-squares functional.
Define the set of admissible solution
	\[
		H = \{(u_m)_{m = 1}^N \in H^2(\Omega)^N: u_m|_{\partial \Omega} = f_m \mbox{ and } \partial_{\nu} u_m|_{\partial \Omega} = g_m, 1 \leq m \leq N\}.
	\] 	
	Throughout the paper, we assume that the set $H$ is nonempty.	
	In the analysis, we will need the following subspace of $H^2(\Omega)^N$ 
		\begin{equation}
			H_0 = \left\{
				(v_1, \dots, v_N) \in H^2(\Omega): 
				v_m(\x) = \partial_{\nu} v_m(\x) = 0
			\right\}.
		\label{H0}
		\end{equation}
Let $\x_0$ be a point in $\R^d \setminus \Omega$ with $\min\{r(\x): \x \in \overline \Omega\} > 1$ and $b > \max\{r(\x): \x \in \overline \Omega\}$ where
\[
	r(\x) = |\x - \x_0| \quad \mbox{for all } \x \in \R^d.
%	\label{rx}
\]
We choose $\x_0$ such that $\min\{r(\x): \x \in \overline \Omega\} > 1.$
To find $u^{(0)}$, we minimize the functional 
$
	J^{(0)}: H \to \R
$ with
\begin{equation}
	J^{(0)}(u_1, \dots, u_N) 
	= 
	\sum_{m = 1}^N\int_{\Omega} e^{2\lambda b^{-\beta} r^\beta(\x)} \Big|\Delta u_m - c(\x) \sum_{n = 1}^N s_{mn} u_n\Big|^2 d\x 
	\label{J0}
\end{equation}
where $\lambda$ and $\beta$ are the numbers as in Corollary \ref{Col 3.1}.
The obtained minimizer $(u_m^{(0)})_{m = 1}^N$ $\in$ $H$ is called the  regularized solution to (\ref{U0}).
Next, assume, by induction, that we know $(u_m^{(k-1)})_{m = 1}^N$, $k \geq 1$, we set $(u_m^{(k)})_{m = 1}^N$ as the minimizer of 
$
	J^{(k)}: H \to \R
$ defined as
\begin{multline}
	J^{(k)}(u_1, \dots, u_N) = \sum_{m = 1}^N\int_{\Omega} e^{2\lambda b^{-\beta} r^\beta(\x)}\Big|\Delta u_m - c(\x)\sum_{n = 1}^N s_{mn} u_n
	\\
 + q_m(P(u_1^{(k - 1)}), \dots, P(u_N^{(k-1)}))\Big|^2 d\x.
 \label{Jk}
\end{multline}

\begin{remark}
	The function $e^{2\lambda b^{-\beta} r^\beta(\x)}$ in \eqref{J0} and \eqref{Jk} is called the Carleman weight function. 
	Its presence is very helpful to prove the existence and uniqueness of the minimizers for the functionals $J^{(k)}$, $k \geq 0,$ see Theorem \ref{minimizer}. 
	On the other hand, this Carleman weight function and the Carleman estimate (see Theorem \ref{Carleman}) play important roles for us to prove the convergence of our method, see Theorem  \ref{main thm}.
\end{remark}

The following result guarantees the existence and uniqueness of the minimizer of (\ref{U0}) and the one of (\ref{Up}), $k \geq 1$.
\begin{Theorem}
	Assume that $f_m$ and $g_m$ are in $L^2(\partial \Omega)$, $m = 1, 2, \dots, N$ and assume that $H$ is nonempty. 	
	Then, each functional $J^{(k)}$, $k \geq 0$, has a unique minimizer provided that both $\lambda$ and $\beta$ are sufficiently large. 
	\label{minimizer}
\end{Theorem}

\begin{proof}
	We only prove Theorem \ref{minimizer} when $k \geq 1$.
	Since $H$ is nonempty, we can find a vector valued function $(\varphi_m)_{m = 1}^N \in H.$ 
	Define 
	\begin{equation}
		v_m(\x) = u_m(\x) - \varphi_m(\x) \quad \x \in \Omega, m = 1, \dots, N.
		\label{3939}
	\end{equation}
	We minimize
	\[
		I^{(k)}(v_1, \dots, v_N) = J^{(k)}(u_1 - \varphi_1, \dots, u_N - \varphi_N) 
	\] where $(v_m)_{m = 1}^N$ varies in $H_0$, defined in (\ref{H0}).
	If $(v_m)_{m = 1}^N$ minimizes $I^{(k)},$ then by the variational principle,
\begin{eqnarray}
	\hspace{-2cm}
	\sum_{m = 1}^N\Big \langle 
	e^{2\lambda b^{-\beta} r^\beta(\x)} \Big(\Delta v_m - c(\x)\sum_{n = 1}^N s_{mn} v_n
	+ \Delta \varphi_m - c(\x)\sum_{n = 1}^N s_{mn} \varphi_n
	\nonumber
	\\
%	\hspace{1cm} 
	+ q_m(P(u_1^{(k - 1)}), \dots, P(u_N^{(k-1)}))\Big),	
	\Delta h_m - c(\x)\sum_{n = 1}^N s_{mn} h_n \Big \rangle_{L^2(\Omega)} = 0
	\label{39}
\end{eqnarray}
for all $(h_m)_{m = 1}^N \in H_0$.
The identity (\ref{39}) is equivalent to
\begin{multline}
%	\hspace{-2cm}
	\sum_{m = 1}^N\Big \langle 
	e^{2\lambda b^{-\beta} r^\beta(\x)} \Delta v_m - c(\x)\sum_{n = 1}^N s_{mn} v_n,
	 \Delta h_m - c(\x)\sum_{n = 1}^N s_{mn} h_n \Big \rangle_{L^2(\Omega)}
%	\nonumber
	\\
%	\hspace{-1cm}
=-\sum_{m = 1}^N\Big \langle
	 e^{2\lambda b^{-\beta} r^\beta(\x)} \Big( \Delta \varphi_m - c(\x)\sum_{n = 1}^N s_{mn} \varphi_n + q_m(P(u_1^{(k - 1)}), \dots, P(u_N^{(k-1)}))\Big),
%	 \nonumber
	 \\
%	 \hspace{5cm}	
	\Delta h_m - c(\x)\sum_{n = 1}^N s_{mn} h_n \Big \rangle_{L^2(\Omega)}. 
	\label{40}
\end{multline}
The left hand side of (\ref{40}) defines a bilinear form  $\{\cdot, \cdot\}$ of a pair $((v_m)_{m = 1}^N, (h_m)_{m = 1}^N)$ in $H_0$.

	We claim that $\{\cdot, \cdot\}$ is coercive; that means,	 \[
		\{(v_m)_{m = 1}^N, (v_m)_{m = 1}^N\} \geq C\|(v_m)_{m = 1}^N\|_{H^2(\Omega)^N}^2
	\] for some constant $C$.
	In fact, using the inequality $(x - y )^2 \geq x^2/2 - y^2 $, we have
\begin{multline*}
%	\hspace{-1cm}	
	\sum_{m = 1}^N\int_{\Omega} e^{2\lambda b^{-\beta} r^\beta(\x)}\Big|\Delta v_m - c(\x)\sum_{n = 1}^N s_{mn} v_n\Big|^2 d\x
	\geq \sum_{m = 1}^N\int_{\Omega} e^{2\lambda b^{-\beta} r^\beta(\x)}|\Delta v_m|^2 d\x	
	\\
	-\sum_{m = 1}^N\int_{\Omega} e^{2\lambda b^{-\beta} r^\beta(\x)}\Big| c(\x)\sum_{n = 1}^N s_{mn} v_n\Big|^2 d\x.
\end{multline*}
Applying the Carleman estimate (\ref{33}), which will be proved in Section \ref{Sec Carleman}, for the function $v_m$ for each $m \in \{1, \dots, N\},$ we have
\begin{multline*}
	%\hspace{-2cm}	
	\sum_{m = 1}^N\int_{\Omega} e^{2\lambda b^{-\beta} r^\beta(\x)}\Big|\Delta v_m - c(\x)\sum_{n = 1}^N s_{mn} v_n\Big|^2 d\x
	\\
	%\hspace{-1cm}	
	\geq \sum_{m = 1}^N\int_{\Omega} e^{2\lambda b^{-\beta} r^\beta(\x)}\Big[\frac{C\sum_{i, j = 1}^d|\partial^2_{x_i x_j} v_m|^2}{\lambda} + 
	C\lambda |\nabla v_m|^2 + C\lambda^3|v_m|^2
	\Big]d\x
	\\	
	%\hspace{6cm}
	-\sum_{m = 1}^N\int_{\Omega} e^{2\lambda b^{-\beta} r^\beta(\x)}\Big| c(\x)\sum_{n = 1}^N s_{mn} v_n\Big|^2 d\x.
\end{multline*}
Since $c(\x)$ and $s_{mn}$ are finite, we can choose $\lambda$ sufficiently large such that
\begin{equation*}
	%\hspace{-1cm}	
	\sum_{m = 1}^N\int_{\Omega} e^{2\lambda b^{-\beta} r^\beta(\x)}\Big|\Delta v_m - c(\x)\sum_{n = 1}^N s_{mn} v_n\Big|^2 d\x
	%\\
%	\hspace{4cm}
	\geq
	C\max_{\x \in \overline \Omega} \{e^{2\lambda b^{-\beta} r^\beta(\x)}\} \lambda^{-1}\sum_{m = 1}^N
	\|(v_m)_l\|^2_{H^2(\Omega)}.
\end{equation*}
Applying the Lax-Milgram theorem, we can find a unique vector-valued function $(v_m)_{m = 1}^N$ satisfying (\ref{40}). 
The vector-valued function $(u_m)_{m = 1}^N$ can be found via (\ref{3939}).
\end{proof}

Denote by $\{(u_1^{(k)}, \dots, u_N^{(k)})\}$ the sequence of the minimize of $J^{(k)}$, $k \geq 0$ .
We claim that this sequence converges to the true solution of (\ref{system U}) and (\ref{boundary conditions}) in $L^2(\Omega)^N$ as $k \to \infty$.
The proof of this fact is based on the contraction principle and the Carleman estimate in Section \ref{Sec Carleman} plays an important role.

\section{A new Carleman estimate} \label{Sec Carleman}

In this section, we establish a new Carleman estimate, which has been used  to prove Theorem \ref{minimizer} that guarantees the unique existence of the functional $J^{(k)}$ in \eqref{Jk}, $k \geq 1$.
This estimate and its corollary, Corollary \ref{Col 3.1}, play a crucial role in the proof of our main result,  see Theorem \ref{main thm} which guarantees the convergence of our numerical method. 
\begin{Theorem}[Carleman estimate]
	Let $\x_0$ be a point in $\R^d \setminus \overline \Omega$ such that $r(\x) = |\x - \x_0| > 1$ for all $\x \in \Omega$.
	Let $b > \max_{\x \in \overline \Omega} r(\x)$ be a fixed constant.
	There exist positive constants $\beta_0$  depending only on $b$, $\x_0$, $\Omega$ and $d$ such that
	for all function $v \in C^2(\overline \Omega)$ satisfying 
	 \[
	 	v(\x) = \partial_{\nu} v(\x) = 0  \quad \mbox{for all } \x \in \partial \Omega,
	 \]
	the following estimate holds true
\begin{multline}
	\int_{\Omega} e^{2\lambda b^{-\beta} r^{\beta}(\x)}|\Delta v(\x)|^2 d\x
	\geq 
	\frac{C}{\lambda \beta^{7/4} b^{-\beta}}\sum_{i, j = 1}^d \int_{\Omega}e^{2\lambda b^{-\beta} r^\beta(\x)} r^{2\beta}(\x)|\partial^2_{ x_i x_j}v(\x)|^2 d\x
	\\
	%\hspace{-1cm}
	+ 	C \lambda^3 \beta^4 b^{-3 \beta} \int_{\Omega} r^{2\beta}(\x) e^{2\lambda b^{-\beta} r^{\beta}}|v(\x)|^2 d\x
	\\
	+ C \lambda \beta^{1/2} b^{-\beta}\int_{\Omega} e^{2\lambda b^{-\beta} r^{\beta}(\x)} |\nabla v(\x)|^2 d\x
	\label{Car est}
\end{multline}
	for $\beta \geq \beta_0$ and $\lambda \geq \lambda_0$.
		Here, $\lambda_0 = \lambda_0(b, \Omega, d, \x_0) > 1$ is a positive number with $\lambda_0 b^{-\beta} \gg 1$ and $C = C(b, \Omega, d, \x_0) > 1$ is a constant.  These numbers depend only on listed parameters.
	 \label{Carleman}
\end{Theorem}	
\begin{remark}
	The Carleman estimate in Theorem \ref{Carleman} is more complicated than the version in \cite{NguyenLiKlibanov:IPI2019}.
	The reason for us to establish this new estimate is that the Carleman weight function in \cite{NguyenLiKlibanov:IPI2019} decays fast when $\lambda \gg 1$, causing poor numerical results.
	Unlike this, the Carleman estimate in Theorem \ref{Carleman} allows us to choose large $\lambda$ in implementation, making the theory and the computational codes more consistent.
\end{remark}

\begin{remark}
	A new feature in Theorem \ref{Carleman} the presence of all second derivatives of the function $v$ on the right-hand side of \eqref{Car est}. 
	This makes it more convenient for us to prove the existence and uniqueness of the regularized solutions to a system of nonlinear elliptic equations appearing in our analysis in Section \ref{sec method}, see Theorem \ref{minimizer}.
\end{remark}

We split the proof of Theorem \ref{Carleman} into four lemmas, Lemma \ref{lem 3.4}--Lemma \ref{lem 3.3}.

\begin{Lemma} Let $v$ be the function as in Theorem \ref{Carleman}.
There exists a positive constant $\beta_0$  depending only on $b$, $\x_0$, $\Omega$ and $d$ such that
\begin{multline}
%\hspace{-2cm}
\int_{\Omega}\frac{e^{2\lambda b^{-\beta} r^\beta(\x)}|\Delta v(\x)|^2}{4 \lambda \beta b^{-\beta} r^{\beta - 2}(\x)} d\x
%\nonumber
	%\\
	%\hspace{-1cm}
	\geq 
	C\lambda^2 \beta^3 b^{-2\beta} \int_{\Omega} r^{2\beta}(\x) e^{2\lambda b^{-\beta} r^{\beta}(\x)}|v(\x)|^2 d\x 
	\\
	- C \int_{\Omega} e^{2\lambda b^{-\beta} r^{\beta}(\x)} |\nabla v(\x)|^2 d\x
	\label{4.14}
\end{multline}
for all $\beta \geq \beta_0$ and $\lambda \geq \lambda_0$. Here, $\lambda_0$ is a constant such that $\lambda_0 b^{-\beta} \gg 1$.
\label{lem 3.4}
\end{Lemma}
\begin{proof}
By changing variables, if necessary, we can assume that $\x_0 = 0.$
Define the function 
\begin{equation}
	w(\x) = e^{\lambda b^{-\beta} r^\beta(\x)} v(\x) 
	\quad \mbox{or} \quad
	v(\x) =  e^{-\lambda b^{-\beta} r^\beta(\x)} w(\x) 
	\label{5.11}
\end{equation}
for all $\x \in \Omega$.
Since $v$ vanishes on $\partial \Omega$, so does $w$. 
On the other hand, by the product rule in differentiation, for all $\x \in \Omega$,
\begin{equation}
	\nabla v(\x)= e^{-\lambda b^{-\beta} r^\beta(\x)} \nabla w(\x) -\beta \lambda b^{-\beta} r^{\beta - 2}(\x)  e^{-\lambda b^{-\beta} r^\beta(\x)} w(\x)\x 
	\label{5.1}
\end{equation}
It follows that
\[
	e^{-\lambda b^{-\beta} r^\beta(\x)} \nabla w(\x) \cdot \nu = \nabla v(\x)\cdot \nu  
	+ \beta\lambda b^{-\beta} r^{\beta - 2}(\x)  e^{-\lambda b^{-\beta} r^\beta(\x)} w(\x)\x = 0.
\]  
 for all $\x \in \partial \Omega$.
We thus obtain 
$
	w(\x) = \partial_\nu  w(\x) = 0 
	%\quad \mbox{for all } \x \in \partial \Omega.
%	\label{v in H0}
$ for all $\x \in \partial \Omega$.
Hence, from now on, whenever we apply the integration by parts formula on $v$ and $w$, the integrals on $\partial \Omega$ vanishes.
We next compute the Laplacian of $v$ in terms of $w$. 
For all $\x \in \Omega,$
\begin{align*}
	%\hspace{-1cm}
	\Delta v(\x) &= e^{-\lambda b^{-\beta} r^\beta(\x)} \Delta w(\x) 
	+ 2 \nabla e^{-\lambda b^{-\beta} r^\beta(\x)} \cdot \nabla w(\x) 
	+ w(\x) \Delta (e^{-\lambda b^{-\beta} r^\beta(\x)})
	\\
	& = e^{-\lambda b^{-\beta} r^\beta(\x)} \Big[
		\Delta w(\x)
		- 2  \lambda \beta b^{-\beta} r^{\beta - 2}(\x)\nabla w(\x) \cdot 
\x
\\
& \hspace{6cm}
 				+e^{\lambda b^{-\beta} r^\beta(\x)} \Delta (e^{-\lambda b^{-\beta} r^\beta(\x)}) w(\x)
	\Big].
\end{align*}
Using the inequality $(a - b + c)^2 \geq -2ab - 2bc,$ we have
\begin{multline}
	%\hspace{-1cm}
	|\Delta v(\x)|^2 \geq -4 \lambda \beta b^{-\beta} r^{\beta - 2}(\x) e^{-2\lambda b^{-\beta} r^\beta(\x)}
	\Big[
		\Delta w(\x)\nabla w(\x) \cdot \x
%		\nonumber
		\\		
	%\hspace{4.5cm}	
	+ e^{\lambda b^{-\beta} r^\beta(\x)} \Delta (e^{-\lambda b^{-\beta} r^\beta(\x)})w(\x) \nabla w(\x) \cdot \x 
	\Big] 
\label{4.4}
\end{multline} 
for all $\x \in \Omega.$
By a straight forward computation, for $\x \in \Omega,$
\[
	\Delta(e^{-\lambda b^{-\beta} r^\beta(\x)}) = - \lambda \beta  b^{-\beta} e^{-\lambda b^{-\beta} r^{\beta}(\x)} r^{\beta - 2}(\x) \big[
	(\beta - 2 + d)	 - \lambda b^{-\beta} \beta r^\beta(\x)  
	\big].
\]
Plugging this into (\ref{4.4}) gives
\begin{multline*}
	%\hspace{-1cm}
	|\Delta v(\x)|^2 \geq -4 \lambda  \beta b^{-\beta} r^{\beta - 2}(\x) e^{-2\lambda b^{-\beta} r^\beta(\x)}
	\Big[
		\Delta w(\x)\nabla w(\x) \cdot \x
				\\
	%\hspace{2cm}	
	 -\lambda \beta b^{-\beta} r^{\beta - 2}(\x) \big[
	(\beta - 2 + d)	 - \lambda \beta b^{-\beta} r^\beta(\x)  
	\big] w(\x)\nabla w(\x) \cdot \x 
	\Big] 
\end{multline*} 
for all $\x \in \Omega.$
Hence,
\begin{equation}
	\int_{\Omega}\frac{e^{2\lambda b^{-\beta} r^\beta(\x)}|\Delta v(\x)|^2}{4 \lambda \beta  b^{-\beta} r^{\beta - 2}(\x)} d\x
	\geq I_1 + I_2 + I_3 
	\label{4.5}
\end{equation}
where
\begin{eqnarray}
	I_1 &=& -\int_{\Omega}\Delta w(\x) \nabla w(\x) \cdot  \x d\x, \label{4.6}
	\\
	I_2 &=& \lambda \beta b^{-\beta} (\beta - 2 + d) \int_{\Omega}r^{\beta - 2}(\x) w(\x) \nabla w(\x) \cdot \x d\x, \label{4.7}
	\\
	I_3 & =& - \lambda^2 b^{-2\beta} \beta^2 \int_{\Omega}r^{2\beta - 2}(\x) w(\x) \nabla w(\x) \cdot \x d\x. \label{4.8}
\end{eqnarray}

\noindent{\bf Estimate $I_1$.}
Write $\x = (x_1, \dots, x_d)$  and integrating $I_1$ by parts. 
It follows from (\ref{4.6}) that $I_1$ is equal to
\begin{align*}
	%\hspace{-2cm}
	 \int_{\Omega} \nabla w(\x) &\cdot \nabla [\nabla w(\x) \cdot \x] d\x
	 \\
	&
	= \sum_{i, j = 1}^d\int_{\Omega}  \partial_{x_i} w(\x) \partial_{x_i} x_j\partial_{x_j} w(\x))d\x
	\\
	& = \sum_{i, j = 1}^d \int_{\Omega} \partial_{x_i} w(\x) [\partial_{x_j} w(\x) \delta_{ij} + x_j \partial_{x_i x_j} w(\x)] d\x
	\\
	& = \sum_{i = 1}^d \int_{\Omega} |\partial_{x_i} w(\x)|^2 d\x 
	+ \sum_{i, j = 1}^d \int_{\Omega} x_j\partial_{x_i} w(\x) \partial_{x_j x_i} w(\x)  d\x.
\end{align*}
Using the identity $\phi(\x) \partial_{x_j} \phi(\x) = \frac{1}{2} \partial_{x_j} (\phi(\x)^2)$ with $\Phi(\x) = \partial_{x_i}w(\x)$ gives
\begin{eqnarray*}
	I_1 &=& \int_{\Omega} |\nabla w(\x)|^2 d\x + \frac{1}{2} \sum_{i, j = 1}^d \int_{\Omega} x_j\partial_{x_j} (\partial_{x_i} w(\x))^2 d\x
	\\
	 &=& \int_{\Omega} |\nabla w(\x)|^2 d\x - \frac{1}{2} \sum_{i, j = 1}^d \int_{\Omega}  (\partial_{x_i} w(\x))^2\partial_{x_j} x_j d\x.
\end{eqnarray*}
Hence, 
\begin{equation}
	I_1 = \Big(1 - \frac{d}{2}\Big) \int_{\Omega}|\nabla w(\x)|^2d\x.
\label{4.9}
\end{equation}

\noindent{\bf Estimate $I_2$.} 
We apply the identity $w \nabla w = \frac{1}{2} \nabla |w|^2$ to get from (\ref{4.7})
\begin{eqnarray*}
	I_2 &=& \frac{\lambda \beta b^{-\beta} (\beta - 2 + d)}{2} \int_{\Omega}r^{\beta - 2}(\x)  \nabla |w(\x)|^2 \cdot \x d\x\\
	&=& -\frac{\lambda \beta b^{-\beta} (\beta - 2 + d)}{2} \int_{\Omega}  |w(\x)|^2 \Div (r^{\beta - 2}(\x) \x) d\x.
\end{eqnarray*}
Here, the integration by parts formula was used.
We; therefore, obtain 
\begin{equation}
	I_2 = -\frac{\lambda \beta b^{-\beta} (\beta - 2 + d)^2}{2} \int_{\Omega}  |w(\x)|^2  d\x.
	\label{4.10}
\end{equation}

%
%We apply the inequality $ab \leq a^2/2 + b^2/2$. It follows from (\ref{4.7}) that
%\begin{eqnarray*}
%	|I_2| 
%	& \leq&    \int_{\Omega} |\lambda \beta b^{-\beta} r^{\beta - 2}(\x) w(\x)| \big(|\beta - 2 + d|  |\nabla w(\x)|\big) d\x
%	\\
%	&\leq&  \lambda^2 \beta^2 b^{-2\beta} \int_{\Omega} r^{2\beta - 4}(\x)|w(\x)|^2 d\x
%	+ \int_{\Omega} (\beta - 2 + d)^2  \nabla w(\x)|^2 d\x.
%\end{eqnarray*}
%Since  $\beta$ is large, 
%\begin{equation}
%	I_2 \geq - C\lambda^2 \beta^2 b^{-2\beta}  \int_{\Omega} r^{2\beta}(\x) |w(\x)|^2 d\x - C \beta^2 \int_{\Omega}|\nabla w(\x)|^2 d\x
%	\label{4.10}
%\end{equation}
%where $C$ is a generous constant depending on  $d$, $\Omega$ and $\x_0$.

\noindent{\bf Estimate $I_3.$} Using integration by parts formula again, by (\ref{4.8}), 
\begin{eqnarray*}
	I_3 &=& - \frac{\lambda^2 \beta^2 b^{-2\beta}}{2} \int_{\Omega}r^{2\beta - 2}(\x)  \nabla |w(\x)|^2 \cdot \x d\x
	\\
	&=& \frac{\lambda^2 \beta^2 b^{-2\beta}}{2} \int_{\Omega}  |w(\x)|^2 \Div[r^{2\beta - 2}(\x)\x] d\x.
\end{eqnarray*}
Hence,
\begin{eqnarray}
	I_3 &=& \frac{\lambda^2  \beta^2(2\beta - 2 + d) b^{-2\beta}}{2} \int_{\Omega}  |w(\x)|^2 r^{2\beta -2 }(\x) d\x
	\nonumber
	\\
	&\geq& C \lambda^2 \beta^3 b^{-2\beta}\int_{\Omega} r^{2\beta}(\x)|w(\x)|^2 d\x.
	\label{4.11}
\end{eqnarray}

Combining (\ref{4.5}), (\ref{4.9}), (\ref{4.10}) and (\ref{4.11}) and using the fact that $\lambda_0 b^{-\beta} \gg 1$ (which implies $\lambda b^{-\beta} \gg 1$), we get
\begin{equation}
	\int_{\Omega}\frac{e^{2\lambda b^{-\beta} r^\beta(\x)}|\Delta v(\x)|^2}{4\lambda \beta b^{-\beta}  r^{\beta - 2}(\x)} d\x 
	%\\
	\geq 
	C\lambda^2 \beta^3 b^{-2\beta} \int_{\Omega} r^{2\beta}(\x) |w(\x)|^2 d\x - C \int_{\Omega} |\nabla w(\x)|^2 d\x.
	\label{4.12}
\end{equation}
Recall (\ref{5.11}) that $w = e^{\lambda b^{-\beta} r^{\beta}} v$. 
We have for all $\x \in \Omega$,
\begin{equation}
	\nabla w(\x) 
	%= \nabla (e^{\lambda r^\beta(\x)} v(\x)) 
%	= e^{\lambda r^\beta(\x)} \nabla v(\x) +  v(\x)\nabla (e^{\lambda r^\beta(\x)})
	= e^{\lambda b^{-\beta} r^\beta(\x)} [\nabla v(\x) + \lambda b^{-\beta} \beta r^{\beta - 2}(\x)v(\x) \x]. 
	\label{4.13}
\end{equation}
It follows from (\ref{5.11}), (\ref{4.12}), (\ref{4.13}), the triangle inequality and the fact $\beta^3 \gg \beta^2$ that
\begin{equation*}
		%\hspace{-2cm}
		\int_{\Omega}\frac{e^{2\lambda b^{-\beta} r^\beta(\x)}|\Delta v(\x)|^2}{4 \lambda \beta b^{-\beta} r^{\beta - 2}(\x)} d\x
	%\\
	%\hspace{-1cm}
	\geq 
	C\lambda^2 \beta^3 b^{-2\beta} \int_{\Omega} r^{2\beta}(\x) e^{2\lambda b^{-\beta} r^{\beta}(\x)}|v(\x)|^2 d\x 
	%\\
	- C \int_{\Omega} e^{2\lambda b^{-\beta} r^{\beta}(\x)} |\nabla v(\x)|^2 d\x.
%	\label{4.14}
\end{equation*}
Recall that $\rho = \max_{\x \in \overline \Omega} r(
x)$. 
We have obtained the desired inequality (\ref{4.14}).
\end{proof}

\begin{Lemma}
Let $v$ be the function that satisfying all hypotheses of Theorem \ref{Carleman}.
There exist positive constants $\beta_0$ and $\lambda_0$  depending only on $b$, $\x_0$, $\Omega$ and $d$ such that
	\begin{multline}
	%	\hspace{-2cm}
		-\int_{\Omega} e^{2\lambda b^{-\beta} r^{\beta}(\x)} v(\x) \Delta v(\x) d\x
	%	\nonumber
	%	\\
	%	\hspace{-1cm}
		\geq 
		C \int_{\Omega} e^{2\lambda b^{-\beta} r^{\beta}(\x)} |\nabla v(\x)|^2 d\x
		\\
	- C\lambda^2 \beta^2 b^{-2\beta}\int_{\Omega} e^{2\lambda b^{-\beta}r^{2\beta}(\x)} r^{2\beta}(\x)|v(\x)|^2d\x
		\label{4.15}
	\end{multline}
for all $\beta \geq \beta_0$ and $\lambda \geq \lambda_0.$
\label{lem 3.2}
\end{Lemma}

\begin{proof}
	By integrating by parts, we have
	\begin{align}
		\int_{\Omega} &e^{2\lambda b^{-\beta} r^{\beta}(\x)} v(\x) \Delta v(\x) d\x
		\nonumber
		\\
		&= 
		\int_{\Omega}  \nabla v(\x) \cdot \nabla \big( e^{2\lambda b^{-\beta} r^{\beta}(\x)} v(\x) \big)d\x 
		\nonumber
		\\
		&= \int_{\Omega} e^{2\lambda b^{-\beta} r^{\beta}(\x)} |\nabla v(\x)|^2 d\x
		+  \int_{\Omega} v(\x)\nabla v(\x) \cdot 
		\nabla \big( e^{2\lambda b^{-\beta} r^{\beta}(\x)}\big)d\x.
		\label{27}
	\end{align}
	The absolute value of second integral in the right hand side of (\ref{27}) can be estimated as
	\begin{align}
	%\hspace{-2cm}
	\Big|\int_{\Omega} & v(\x)\nabla v(\x) \cdot 
		\nabla \big( e^{2\lambda b^{-\beta} r^{\beta}(\x)}\big)d\x\Big| \nonumber
		\\
		&\leq  2\lambda \beta b^{-\beta}\int_{\Omega}  r^{\beta - 1}(\x)e^{2\lambda b^{-\beta} r^{\beta}(\x)}|v(\x)||\nabla v(\x)| 
		 d\x
		 \nonumber
		 \\
		 &\leq 
		 C\lambda^2 \beta^2 b^{-2\beta}\int_{\Omega} e^{2\lambda b^{-\beta}r^{\beta}(\x)} r^{2\beta}(\x)|v(\x)|^2d\x
	%	\nonumber
%		\\
%		&\hspace{4cm}
		+ \frac{1}{2} \int_{\Omega}  e^{2\lambda b^{-\beta}r^{\beta}(\x)} |\nabla v(\x)|^2 d\x.
		\label{2828}
	\end{align}
	%Since $r(\x) > 1$, $r^{-\beta - 2}(\x) \gg r^{-2\beta}(\x)$ for all $\x \in \Omega$. 
	This, (\ref{27})  and (\ref{2828}) imply
	\begin{multline*}
		%\hspace{-2cm}
		-\int_{\Omega} e^{2\lambda b^{-\beta} r^{\beta}(\x)} v(\x) \Delta v(\x) d\x
		%\\
		\geq 
		C \int_{\Omega} e^{2\lambda b^{-\beta} r^{\beta}(\x)} |\nabla v(\x)|^2 d\x
%		\\
%		\hspace{2cm}- 
\\
		-C\lambda^2 \beta^2 b^{-2\beta}\int_{\Omega} e^{2\lambda b^{-\beta}r^{2\beta}(\x)} r^{2\beta}(\x)|v(\x)|^2d\x.
	\end{multline*}
	The lemma is proved.
\end{proof}

\begin{Lemma}
Let $v$ be the function satisfying all hypotheses of Theorem \ref{Carleman}.
There exist positive constants $\beta_0$  depending only on $b$, $\x_0$, $\Omega$ and $d$ such that
%\begin{eqnarray}
%		\hspace{-2cm}
%		 \int_{\Omega} \frac{e^{2\lambda b^{-\beta} r^{\beta}(\x)}}{\lambda \beta b^{-\beta}  r^{\beta - 2}(\x)} |\Delta v(\x)|^2 d\x
%		 \geq C \lambda^2 \beta^3 b^{-2 \beta} \int_{\Omega} r^{2\beta}(\x) e^{2\lambda b^{-\beta} r^{\beta}}|v(\x)|^2 d\x
%		 \nonumber
%		 \\
%		\hspace{6cm} + C \beta^{1/2} \int_{\Omega} e^{2\lambda b^{-\beta} r^{\beta}(\x)} |\nabla v(\x)|^2 d\x.	
%		\label{29}
%	\end{eqnarray}
	\begin{multline}
	%\hspace{-2cm}
	\int_{\Omega} e^{2\lambda b^{-\beta} r^{\beta}(\x)}|\Delta v(\x)|^2 d\x 
	%\nonumber
	\geq 
	C \lambda^3 \beta^4 b^{-3 \beta} \int_{\Omega} r^{2\beta}(\x) e^{2\lambda b^{-\beta} r^{\beta}}|v(\x)|^2 d\x
	\\
		%\hspace{5cm} 
		+ C \lambda \beta^{1/2} b^{-\beta}\int_{\Omega} e^{2\lambda b^{-\beta} r^{\beta}(\x)} |\nabla v(\x)|^2 d\x	
		\label{29}
	\end{multline}
for all $\beta \geq \beta_0$ and $\lambda \geq \lambda_0.$ 
Here $\lambda_0$ is a constant satisfying $\lambda_0 b^{-\beta} > 1.$
\label{lem 3.2}
\end{Lemma}

\begin{proof}
	Multiplying $\beta^{1/4}$ to (\ref{4.15}) and then applying the inequality $-ab \leq a^2/2 + b^2/2$, we have
	\begin{multline*}
		%\hspace{-2cm}
		\int_{\Omega} \lambda \beta^{3/2} b^{-\beta}e^{2\lambda b^{-\beta} r^{\beta}(\x)} r^{\beta - 2}(\x) |v(\x)|^2 d\x
		+ 
		\int_{\Omega} \frac{e^{2\lambda b^{-\beta} r^{\beta}(\x)}}{4\lambda b^{-\beta} \beta r^{\beta - 2}(\x)} |\Delta v(\x)|^2 d\x
		\\
		\geq 
		C \beta^{1/2} \int_{\Omega} e^{2\lambda b^{-\beta} r^{\beta}(\x)} |\nabla v(\x)|^2 d\x
		\\
	- C\lambda^2 \beta^{5/2} b^{-2\beta}\int_{\Omega}  
			 r^{2\beta}(\x) 
		 e^{2\lambda b^{-\beta}r^{\beta}(\x)}|v(\x)|^2d\x.
	\end{multline*}
	Since $r(\x) > 1,$ $\beta^{3/2} r^{\beta - 2}(\x) \ll r^{2\beta}(\x)$, 
	we have
	\begin{multline}
		%\hspace{-2cm}
		\int_{\Omega} \frac{e^{2\lambda b^{-\beta} r^{\beta}(\x)}}{4\lambda \beta b^{-\beta}  r^{\beta - 2}(\x)} |\Delta v(\x)|^2 d\x
		%\nonumber		
		\geq 
		C \beta^{1/2} \int_{\Omega} e^{2\lambda b^{-\beta} r^{\beta}(\x)} |\nabla v(\x)|^2 d\x
		\\
		%		\hspace{4cm}
	- C\lambda^2 \beta^{5/2} b^{-2\beta}\int_{\Omega}  
			 r^{2\beta}(\x) 
		 e^{2\lambda b^{-\beta}r^{\beta}(\x)}|v(\x)|^2d\x.
		 \label{3030}
	\end{multline}
	Here, we have used the fact that $\lambda b^{-\beta} \gg 1.$
	Adding (\ref{3030}) and (\ref{4.14}) together, we obtain
	\begin{multline*}
		%\hspace{-2cm}
		 \int_{\Omega} \frac{e^{2\lambda b^{-\beta} r^{\beta}(\x)}}{\lambda \beta b^{-\beta}  r^{\beta - 2}(\x)} |\Delta v(\x)|^2 d\x
		 \geq C \lambda^2 \beta^3 b^{-2 \beta} \int_{\Omega} r^{2\beta}(\x) e^{2\lambda b^{-\beta} r^{\beta}}|v(\x)|^2 d\x
%		 \nonumber
		 \\
		%\hspace{6cm} 
		+ C \beta^{1/2} \int_{\Omega} e^{2\lambda b^{-\beta} r^{\beta}(\x)} |\nabla v(\x)|^2 d\x,
	\end{multline*}
	which implies (\ref{29}).
\end{proof}

\begin{Lemma}
Let $v$ be the function satisfying all hypotheses of Theorem \ref{Carleman}.
There exist positive constants $\beta_0$ and $\lambda_0$  depending only on $b$, $\x_0$, $\Omega$ and $d$ such that
\begin{multline}
	%\hspace{-2cm}
	\frac{1}{\lambda \beta^{7/4} b^{-\beta}}\int_{\Omega}e^{2\lambda b^{-\beta} r^\beta(\x)}|\Delta v(\x)|^2d\x
\\
	%\nonumber
	\geq 
	\frac{C}{\lambda \beta^{7/4} b^{-\beta}}\sum_{i, j = 1}^d \int_{\Omega}e^{2\lambda b^{-\beta} r^\beta(\x)} r^{2\beta}(\x)|\partial^2_{ x_i x_j}v(\x)|^2 d\x 
	\\
	%\hspace{5cm}
	- C \lambda \beta^{1/4} b^{-\beta} \int_\Omega e^{2\lambda b^{-\beta} r^\beta(\x)} |\nabla v(\x)|^2d\x
	\label{28}
\end{multline}
for all $\beta \geq \beta_0$ and $\lambda \geq \lambda_0.$
\label{lem 3.3}
\end{Lemma}
\begin{proof}
	By the density arguments, we can assume that $v \in C^3(\overline \Omega).$
	Write $\x = (x_1, \dots, x_d)$. 
	We have
\begin{align*}
%		\hspace{-1cm}
		\int_{\Omega}e^{2\lambda b^{-\beta} r^\beta(\x)}|\Delta v(\x)|^2d\x 
		&= 
		\sum_{i, j = 1}^d\int_{\Omega}e^{2\lambda b^{-\beta} r^\beta(\x)} \partial^2_{x_i x_i} v(\x) \partial^2_{x_j x_j} v(\x) d\x
		\\
		%\hspace{-1cm}
		&=
		\sum_{i, j = 1}^d
			\int_{\Omega}\partial_{x_j}\Big[e^{2\lambda b^{-\beta} r^\beta(\x)} \partial^2_{x_i x_i} v(\x) \partial_{x_j} v(\x) \Big]d\x
\\
	&\hspace{1cm} 
	- \sum_{i, j = 1}^d
			\int_{\Omega} \partial_{x_j} v(\x)\partial_{x_j}\Big[e^{2\lambda b^{-\beta} r^\beta(\x)} \partial^2_{x_i x_i} v(\x)  \Big]d\x. 
\end{align*}
The first  integral in the right hand side above vanishes due to the divergence theorem.
Hence
\begin{multline}
%\hspace{-1cm}
\int_{\Omega}e^{2\lambda b^{-\beta} r^\beta(\x)}|\Delta v(\x)|^2d\x  
%\nonumber
	=
	-\sum_{i, j = 1}^d \int_{\Omega}e^{2\lambda b^{-\beta} r^\beta(\x)} \partial_{x_j} v(\x)  \partial^3_{x_i x_i x_j}v(\x) d\x 
	%\nonumber
	\\
	%\hspace{4.5cm}
	-\sum_{i, j = 1}^d \int_{\Omega}\partial_{x_j}(e^{2\lambda b^{-\beta} r^\beta(\x)}) \partial_{x_j} v(\x)  \partial^2_{x_i x_i}v(\x) d\x.
	\label{3.18}
\end{multline}
	The first term in the right hand side of (\ref{3.18}) is rewritten as
\begin{align*}
	%\hspace{-2cm}
	-\sum_{i, j = 1}^d &\int_{\Omega}e^{2\lambda b^{-\beta} r^\beta(\x)} \partial_{x_j} v(\x)  \partial^3_{x_i x_i x_j}v(\x) d\x 
	\\
	&=
	\sum_{i, j = 1}^d \int_{\Omega}\partial_{x_i}(e^{2\lambda b^{-\beta} r^\beta(\x)} \partial_{x_j} v(\x) ) \partial^2_{ x_i x_j}v(\x) d\x\\
	%\hspace{-1cm}
	&= \sum_{i, j = 1}^d \int_{\Omega}e^{2\lambda b^{-\beta} r^\beta(\x)} |\partial^2_{ x_i x_j}v(\x)|^2 d\x
	\\
	&\hspace{3cm}
	+\sum_{i, j = 1}^d \int_{\Omega}\partial_{x_j} v(\x)\partial_{x_i}(e^{2\lambda b^{-\beta} r^\beta(\x)}  ) \partial^2_{ x_i x_j}v(\x) d\x.
\end{align*}
	Combining this and (\ref{3.18}), we have
\begin{multline*}
	%\hspace{-2cm}
	\int_{\Omega}e^{2\lambda b^{-\beta} r^\beta(\x)}|\Delta v(\x)|^2d\x
	=\sum_{i, j = 1}^d \int_{\Omega}e^{2\lambda  b^{-\beta}r^\beta(\x)} |\partial^2_{ x_i x_j}v(\x)|^2 d\x 
	\\
	%\hspace{-1cm}
	+ \sum_{i, j = 1}^d \int_{\Omega}
	\Big[
		\partial_{x_j} v(\x)\partial_{x_i}(e^{2\lambda b^{-\beta} r^\beta(\x)}  ) \partial^2_{ x_i x_j}v(\x)
		\\
		-\partial_{x_j}(e^{2\lambda b^{-\beta} r^\beta(\x)}) \partial_{x_j} v(\x)  \partial^2_{x_i x_i}v(\x) d\x
	\Big].
\end{multline*}
Hence,
\begin{multline*}
	\int_{\Omega}e^{2\lambda b^{-\beta} r^\beta(\x)}|\Delta v(\x)|^2d\x
	\geq \sum_{i, j = 1}^d \int_{\Omega}e^{2\lambda b^{-\beta} r^\beta(\x)} |\partial^2_{ x_i x_j}v(\x)|^2 d\x 
	\\
	%\hspace{4cm}
	- 2\sum_{i, j = 1}^d \int_{\Omega}
		|\partial_{x_j} v(\x)||\partial_{x_i}(e^{2\lambda b^{-\beta}r^\beta(\x)}  )| \partial^2_{ x_i x_j}v(\x)|d\x.
\end{multline*}
Note that for all $i = 1, \dots, d$,
\[
	\partial_{x_i}(e^{2\lambda b^{-\beta} r^\beta(\x)}) = 2 \lambda b^{-\beta} \beta r^{\beta - 2} e^{2\lambda b^{-\beta} r^\beta(\x)}x_i \quad \mbox{for all } \x \in \Omega.
\]
Using the inequality $ab \leq a^2/2 + b^2/2$, we obtain (\ref{28}).
\end{proof}

We now prove Theorem \ref{Carleman}.

%\noindent{\it Proof of Theorem \ref{Carleman}:} 
\begin{proof}[Proof of Theorem \ref{Carleman}]
Adding  (\ref{29}) and (\ref{28}) together, we obtain
\begin{multline*}
	%\hspace{-2cm}
	(1 + \frac{1}{\lambda^2 \beta^{7/4} b^{-2\beta}})\int_{\Omega} e^{2\lambda b^{-\beta} r^{\beta}(\x)}|\Delta v(\x)|^2 d\x
	\\
	%\hspace{-1.5cm}
	\geq 
	\frac{C}{\lambda \beta^{7/4} b^{-\beta}}\sum_{i, j = 1}^d \int_{\Omega}e^{2\lambda b^{-\beta} r^\beta(\x)} r^{2\beta}(\x)|\partial^2_{ x_i x_j}v(\x)|^2 d\x
	\\
	%\hspace{-1cm}
	+ 	C \lambda^3 \beta^4 b^{-3 \beta} \int_{\Omega} r^{2\beta}(\x) e^{2\lambda b^{-\beta} r^{\beta}}|v(\x)|^2 d\x
	\\
	+ C \lambda \beta^{1/2} b^{-\beta}\int_{\Omega} e^{2\lambda b^{-\beta} r^{\beta}(\x)} |\nabla v(\x)|^2 d\x.	
\end{multline*}
%Theorem \ref{Carleman} has been proved. 
\end{proof}

\begin{Corollary}
	Recall $\beta_0$ and $\lambda_0$ as in Theorem  \ref{Carleman}.
	Fix $\beta = \beta_0$ and let the constant $C$ depend on $\x_0,$ $\Omega,$ $d$ and $\beta$. 
	There exists a constant $\lambda_0$ depending only on $\x_0,$ $\Omega,$ $d$ and $\beta$ such that
  for all function $v \in H^2(\Omega)$ with 
	\[
		v(\x) = \partial_{\nu} v(\x) = 0 
		\quad \mbox{ on } \partial \Omega,
	\] 
	%on $\partial \Omega,$
	%there exists a constant $\lambda_0$  such that
	we have
	\begin{multline}
	%\hspace{-2cm}
	\int_{\Omega} e^{2\lambda b^{-\beta} r^{\beta}(\x)}|\Delta v(\x)|^2 d\x
	%\nonumber
	%\\
	\geq 
	C\lambda^{-1}\sum_{i, j = 1}^d \int_{\Omega}e^{2\lambda b^{-\beta} r^\beta(\x)} |\partial^2_{ x_i x_j}v(\x)|^2 d\x
	\\
	+ 	C \lambda^3 \int_{\Omega}  e^{2\lambda b^{-\beta} r^{\beta}}|v(\x)|^2 d\x
	%\nonumber
	%\\
	%\hspace{6cm}
	+ C \lambda \int_{\Omega} e^{2\lambda b^{-\beta} r^{\beta}(\x)} |\nabla v(\x)|^2 d\x	
	\label{33}
\end{multline}
	for all $\lambda \geq \lambda_0$.
	\label{Col 3.1}
\end{Corollary}
    
    \begin{remark}
    Although there are many versions of the Carleman estimate available, those versions are either too complicated, not suitable for us to prove Theorem \ref{minimizer} and Theorem \ref{main thm}, or not work in computations. 
    The main ideas of the proof follow from \cite{BeilinaKlibanovBook, BukhgeimKlibanov:smd1981, Klibanov:jiipp2013, Protter:1960AMS, MinhLoc:tams2015}.
    \end{remark}
    
    \begin{remark}
    	The presence of the second derivatives on the right-hand side of (\ref{33}) is a new feature of our Carleman estimate.
	The presence of those second derivatives allows us to prove the existence and uniqueness of the minimizers of the cost functionals in Section \ref{sec solve nonlinear system}.
    \end{remark}

\section{The convergence analysis} \label{sec convergence}

In this section, we prove a theorem that guarantees that the sequence of vector-valued functions, proposed in Section \ref{sec solve nonlinear system}, converges to the true solution to \eqref{system U}--\eqref{boundary conditions}.
This convergence implies that Algorithm \ref{alg} rigorously provides good numerical solutions to Problem \ref{ISP}.

	\begin{Theorem}		
		Assume that problem (\ref{system U})--(\ref{boundary conditions}) has a unique solution $(u_m^*)_{m = 1}^N$. 
		Then, there is a constant $\lambda$ depending only on $\Omega$, $T$, $d$ and $N$
		such that		
	\begin{equation}
	%	\hspace{-1cm}
		  \sum_{m = 1}^N \Big\| e^{\lambda b^{-\beta} r^\beta(\x)}(u_m^{(k)} - u_m^*)\Big\|^2_{L^2(\Omega)} 
		 % \\
		\leq
		\Big[\frac{C}{ \lambda^3}\Big]^{k - 1} \sum_{m = 1}^N\Big\| e^{\lambda b^{-\beta} r^\beta(\x)} (u_m^{(1)} - u_m^*) \Big\|_{L^2(\Omega)}^2 	
		\label{main est}
		\end{equation}
	for $k = 1, 2, \dots$ where $C$ is a constant depending only on $\Omega$, $T$, $M$, $d$, $N$ and $\|q\|_{C^1(\overline \Omega)}$. 
	In particular, if $\lambda$ is large enough such that $0 < C/\lambda^3 < 1,$  $\{u_m^{(k)}\}_{m = 1}^N$ converges to $u_m^{*}$ exponentially. 
	Moreover, with such a $\lambda$, the sequence $(p^{(k)})_{k \geq 1}$ obtained in Step \ref{Step 8} of Algorithm \ref{alg} converges to the true function $p^* = u^*(\x, 0)$ given by \eqref{Fourier u} with $t = 0$. 
	\label{main thm}
	\end{Theorem}

	\begin{proof}

	In the proof, $C$ is a generous constant  that might change from estimate to estimate.

\noindent {\bf Step 1.}{\it\, Establish a priori bound.}
		Recall $H_0$ as in (\ref{H0}).	
		Since $(u_1^{(k)}, \dots, u_N^{(k)})$ is the minimizer of $J^{(k)}$, by the variational principle, for all $h \in H_0$
		\begin{multline}
	%	\hspace{-2cm}
			\sum_{m = 1}^N
			\Big\langle  e^{\lambda b^{-\beta} r^\beta(\x)}\Big[ \Delta u^{(k)} - c(\x)\sum_{n = 1}^N s_{mn} u^{(k)} + q_m(P(u_1^{(k-1)}), \dots, P(u_N^{(k-1)}))\Big], 	
	%		\nonumber
			\\	
	%		\hspace{4cm}	
			e^{\lambda b^{-\beta} r^\beta(\x)}\Big[\Delta h_m - c(\x)\sum_{n = 1}^Ns_{mn} h_m\Big] \Big\rangle_{L^2(\Omega)} 
			= 0.
			\label{3.6}
		\end{multline}
		On the other hand, since $(u_1^*, \dots, u_N^*)$ solves (\ref{system U})--(\ref{boundary conditions}),
		\begin{multline}
		%\hspace{-2cm}
			\sum_{m = 1}^N
			\Big\langle e^{\lambda b^{-\beta} r^\beta(\x)}\Big[ \Delta u^* - c(\x) \sum_{n = 1}^N s_{mn} u^* + q_m(u_1^*, \dots, u_N^*)\Big], 
		%	\nonumber
			\\
		%	\hspace{4cm}
			e^{\lambda b^{-\beta} r^\beta(\x)}\Big[\Delta h_m - c(\x)\sum_{n = 1}^Ns_{mn} h_m\Big] \Big\rangle_{L^2(\Omega)} 
= 0.
			\label{3.7}
		\end{multline}
		It follows from (\ref{3.6}) and (\ref{3.7}) that
		\begin{multline}
		%\hspace{-2cm}
			\sum_{m = 1}^N
			\Big\langle e^{\lambda b^{-\beta} r^\beta(\x)}\Big[\Delta  (u^{(k)} - u^*) - c(\x)\sum_{n = 1}^N s_{mn} (u^{(k)} - u^*)
		%	\nonumber
			\\ 
		%	 \hspace{-1cm}
			+ q_m(P(u_1^{(k-1)}), \dots, P(u_N^{(k-1)}))) - q_m(u_1^*, \dots, u_N^*)\Big],
		%	\nonumber
			\\
		%	\hspace{4cm}
			e^{\lambda b^{-\beta} r^\beta(\x)}\Big[\Delta h_m - c(\x)\sum_{n = 1}^Ns_{mn} h_m \Big]
			\Big\rangle_{L^2(\Omega)} 	
			= 0.
			\label{3.8}
		\end{multline}
Using the test function $h_m = u_m^{(k)} - u^*_m$, $m = 1, \dots, N$, in (\ref{3.8}) and using H\"older's inequality, we have
\begin{multline}
%\hspace{-2cm}
			\sum_{m = 1}^N
			\Big\| e^{\lambda b^{-\beta} r^\beta(\x)}\Big[\Delta  (u^{(k)} - u^*) - c(\x)\sum_{n = 1}^N s_{mn} (u^{(k)} - u^*)\Big]\Big\|^2_{L^2(\Omega)} 
%			\nonumber
			\\
			\leq
			 \sum_{m = 1}^N  \Big\| e^{\lambda b^{-\beta} r^\beta(\x)}\Big[q_m(P(u_1^{(k-1)}), \dots, P(u_N^{(k-1)}))) - q_m(u_1^*, \dots, u_N^*)\Big]\Big\|_{L^2(\Omega)}
%			 \nonumber
			\\
%			\hspace{1cm}
			\times
			\Big\|e^{\lambda b^{-\beta} r^\beta(\x)}\Big[\Delta (u_m^{(k)} - u^*_m) - c(\x)\sum_{n = 1}^Ns_{mn} (u_m^{(k)} - u^*_m) \Big]
			\Big\|_{L^2(\Omega)}	.
			\label{3.7777}
		\end{multline}
	Using the inequality $\sum_{m = 1}^N a_mb_m \leq (\sum_{m = 1}^N a^2_m)^{1/2}( \sum_{m = 1}^N b^2_m)^{1/2}$ for the right hand side of (\ref{3.7777}) and simplying the resulting, we get
	\begin{multline}
	%\hspace{-2cm}
	\sum_{m = 1}^N
			\Big\| e^{\lambda b^{-\beta} r^\beta(\x)}\Big[\Delta  (u_m^{(k)} - u_m^*) - c(\x)\sum_{n = 1}^N s_{mn} (u_m^{(k)} - u_m^*)\Big]\Big\|_{L^2(\Omega)}^2
	%		\nonumber
			\\
	%		\hspace{-1cm}
			\leq
	\sum_{m = 1}^N \Big\| e^{\lambda b^{-\beta} r^\beta(\x)}\Big[q_m(P(u_1^{(k-1)}), \dots, P(u_N^{(k-1)})))
	%\\
	 - q_m(u_1^*, \dots, u_N^*)\Big]\Big\|_{L^2(\Omega)}^2.
	\label{3.9}
	\end{multline}
	
\noindent{\bf Step 2.}{\it\,  Estimate the right hand side of (\ref{3.9}).} 	
	Since $\|u^*(\x, t)\|_{L^{\infty}} \leq M$, we have
	\begin{align*}
	%\hspace{-1cm}
		|u_m^*(\x)| &= \Big|\int_0^T u^*(\x, t) \Psi_m(t) dt\Big|
		\leq \|u^*(\x, t)\|_{L^2(0, T)} \|\|\Psi_m(t)\|_{L^2(0, T)} 
		\\
		&= \Big(\int_0^T |u^*(\x, t)|^2 dx\Big)^{1/2}
		\leq
		 M \sqrt{T} 
	\end{align*}
	for $m = 1, \dots, N.$
	Therefore,
	\[
		\Big|q_m(P(u_1^{(k-1)}), \dots, P(u_N^{(k-1)})) - q_m(u_1^*, \dots, u_N^*)\Big|
		 \leq A_m \sum_{n = 1}^N\big|u_n^{(k-1)} - u_n^* \big|
	\]
	where 
	\[
		A_m = \max\Big\{|\nabla q_m(s_1, \dots, s_N)|: |s_i| \leq M\sqrt{T}, i = 1, \dots, N\Big\} 
		\quad m = 1, \dots, N.
	\]
	Set $A = \sum_{m = 1}^N A_m$.
	The right hand side of (\ref{3.9}) is bounded from above by
	\begin{align}
		%\hspace{-2cm}
		\sum_{m = 1}^N \Big\| e^{\lambda b^{-\beta} r^\beta(\x)} &\Big[q_m(P(u_1^{(k-1)}), \dots, P(u_N^{(k-1)}))) - q_m(u_1^*, \dots, u_N^*)\Big]\Big\|_{L^2(\Omega)}^2	
		\nonumber	
		\\
		%\hspace{-2cm}
		&\leq
		A \sum_{m = 1}^N\Big\| e^{\lambda b^{-\beta} r^\beta(\x)}\big|P(u_m^{(k - 1)}) - u_m^* \big|\Big\|_{L^2(\Omega)}^2	
		\nonumber
		\\
		&\leq
		 A \sum_{m = 1}^N\Big\| e^{\lambda b^{-\beta} r^\beta(\x)}\big|u_m^{(k - 1)} - u_m^* \big|\Big\|_{L^2(\Omega)}^2. 
	\label{3.10}
	\end{align}
Combining (\ref{3.9}) and (\ref{3.10}) gives
	\begin{multline}
	\sum_{m = 1}^N
			\Big\| e^{\lambda b^{-\beta} r^\beta(\x)}\Big[\Delta  (u_m^{(k)} - u_m^*) - c(\x)\sum_{n = 1}^N s_{mn} (u_m^{(k)} - u_m^*)\Big]\Big\|^2_{L^2(\Omega)}
			%\nonumber
			\\
			%\hspace{5.5cm}
			\leq			
	 A \sum_{m = 1}^N\Big\| e^{\lambda b^{-\beta} r^\beta(\x)}\big|u_m^{(k-1)} - u_m^* \big|\Big\|_{L^2(\Omega)}^2. 
	 \label{3.1212}
	\end{multline}
	
\noindent {\bf Step 3.} {\it Estimate the left hand side of (\ref{3.1212}).}	
	Using the inequality $(a - b)^2 \geq a^2/2 - 2 b^2,$ we have
	\begin{multline}
%	\hspace{-2cm}
	\sum_{m = 1}^N
			\Big\| e^{\lambda b^{-\beta} r^\beta(\x)}\Big[\Delta  (u_m^{(k)} - u_m^*) - c(\x)\sum_{n = 1}^N s_{mn} (u_n^{(k)} - u_n^*)\Big]\Big\|^2_{L^2(\Omega)}
%			\nonumber 
	%\\
	%\hspace{-2.5cm}
	\geq
	\sum_{m = 1}^N\frac{1}{2}\Big\| e^{\lambda b^{-\beta} r^\beta(\x)}\Delta  (u_m^{(k)} - u_m^*)\Big\|^2_{L^2(\Omega)}
	%\nonumber
	\\
%	\hspace{3.5cm}
			-2 \sum_{m = 1}^N \Big\| e^{\lambda b^{-\beta} r^\beta(\x)}c(\x)\sum_{n = 1}^N s_{mn} (u_n^{(k)} - u_n^*)\Big\|^2_{L^2(\Omega)}.								
	\label{3.11}
	\end{multline}
	Applying Carleman estimate in Corollary \ref{Col 3.1},  for the function $u_m^{(k)} - u^*$, $m = 1, \dots, N$, we estimate
	\begin{equation}
	\sum_{m = 1}^N\frac{1}{2}\Big\| e^{\lambda b^{-\beta} r^\beta(\x)}\Delta  (u_m^{(k)} - u_m^*)\Big\|^2_{L^2(\Omega)}
	\\
	\geq
	C \lambda^3  \sum_{m = 1}^N \Big\| e^{\lambda b^{-\beta} r^\beta(\x)}(u_m^{(k)} - u_m^*)\Big\|^2_{L^2(\Omega)}.	 
		\label{3.12}
	\end{equation}
	Fix $\lambda \geq \lambda_0$ where $\lambda_0$ is as in Corollary \ref{Col 3.1}.
	It follows from (\ref{3.11}) and (\ref{3.12}) that
	\begin{multline}
	%\hspace{-2cm}
	\sum_{m = 1}^N\frac{1}{2}\Big\| e^{\lambda b^{-\beta} r^\beta(\x)}\Delta  (u_m^{(k)} - u_m^*)\Big\|^2_{L^2(\Omega)}
	%\nonumber
	%\\
			-2 \sum_{m = 1}^N \Big\| e^{\lambda b^{-\beta} r^\beta(\x)}\sum_{n = 1}^N s_{mn} (u_n^{(k )} - u_n^*)\Big\|^2_{L^2(\Omega)}
			\\
	%		\hspace{5.5cm}
	\geq
	C \lambda^3  \sum_{m = 1}^N \Big\| e^{\lambda b^{-\beta}r^\beta(\x)}(u_m^{(k)} - u_m^*)\Big\|^2_{L^2(\Omega)}.
	\label{3.13}
	\end{multline}
	Combining (\ref{3.9}), (\ref{3.10}) and (\ref{3.13}) gives
	\[
		  \sum_{m = 1}^N \Big\| e^{\lambda b^{-\beta} r^\beta(\x)}(u_m^{(k)} - u_m^*)\Big\|^2_{L^2(\Omega)}
		\leq
		\frac{A}{C \lambda^3} \sum_{m = 1}^N\Big\| e^{\lambda b^{-\beta} r^\beta(\x)}(u_m^{(k - 1)} - u_m^*) \big|\Big\|_{L^2(\Omega)}^2. 
	\]
	By induction, 
	we have
	\begin{equation*}
	%\hspace{-1cm}
		  \sum_{m = 1}^N \Big\| e^{\lambda b^{-\beta} r^\beta(\x)}(u_m^{(k)} - u_m^*)\Big\|^2_{L^2(\Omega)}
		  \\
		\leq
		\Big[\frac{A}{C \lambda^{3}}\Big]^{k - 1} \sum_{m = 1}^N\Big\| e^{\lambda b^{-\beta} r^\beta(\x)}(u_m^{(1)} - u_m^*) \big|\Big\|_{L^2(\Omega)}^2. 
	\end{equation*}
	Replacing $A/C$ by the generous constant $C$, we have proved the estimate \eqref{main est}.  
	The convergence of $p^{(k)}$ to $p^*$ as $k \to \infty$ is obvious.
\end{proof}
	
\begin{remark}
	The technique of  using the Carleman estimate to prove Theorem \ref{main thm} is similar to the one in \cite{BAUDOUIN:SIAMNumAna:2017} in which a coefficient inverse problem for hyperbolic equations was considered. 
	We also find that this technique is applicable to solve an inverse source problem for nonlinear parabolic equations \cite{Boulakia:preprint2019} from the boundary and additional internal measurements. 
\end{remark}

\begin{remark}
	The convergence of $\{p^{(k)}\}_{k \geq 1}$ to the true solution to the inverse problem in Theorem \ref{main thm} is numerically confirmed in Section \ref{sec num}. See also Figures \ref{fig m1 error}--\ref{fig m4 error}.
\end{remark}

\section{Numerical implementation} \label{sec num}

For simplicity, we solve the inverse problem in the case $d = 2$. 

\subsection{The forward problem}
We solve the forward problem of Problem \ref{ISP} as follows.
Let $R_1 > R > 0$ be two positive numbers. 
Define the domains
\[
	\Omega_1 = (-R_1, R_1)^2 
	\quad \mbox{and } 
	\quad \Omega = (-R, R)^2.
\]
We approximate (\ref{main eqn})defined on  $\R^d \times (0, T)$ by the following problem defined on $\Omega_1 \times (0, T)$
\begin{equation}
	\left\{
		\begin{array}{rcll}
			c(\x)u_t(\x, t) &=& \Delta u(\x, t) + q(u(\x, t)) &\x \in \Omega_1, t \in (0, T),\\
			u(\x,0) &=& p(\x) & \x \in \Omega_1,\\
			u(\x, t) &=& 0 & \x \in \partial \Omega_1, t \in [0, T].
		\end{array}
	\right.	
\label{main eqn in G}
\end{equation}
In our numerical tests, 
 the function $c$ is given by
 \begin{equation*}
 	%\hspace{-1cm}
		c(x, y) = 1 + 1/30 
	\Big[3(1-3x)^2e^{-9x^2 - (3y+1)^2}  
	\\
%	\hspace{2cm}
   - 10(3x/5 - 27x^3 - 243y^5)e^{-9x^2-9y^2}  
   - 1/3e^{-(3x+1)^2 - 9y^2}\Big] 
   \label{ctrue}
 \end{equation*}
 for $\x = (x, y) \in \Omega.$
 %This function is taken as 
 %the function ``peaks" of Matlab. 
 The range of $c$ is $[0.8, 1.25]$, which is not a perturbation of the constant function $1$.
%Its graph is displayed in Figure \ref{fig c}.
% \begin{figure}[h!]
% 	\begin{center}
%		\includegraphics[width = 0.3\textwidth]{ctrue}
%		\caption{\label{fig c} The graph of the function $c(\x)$ in all of our numerical tests}
%	\end{center}
% \end{figure}

We solve (\ref{main eqn in G}) by the finite difference method using the explicit scheme. 
The data $f(\x, t) = u(\x, t)$ and $g(\x, t) = \partial_{\nu} u(\x, t)$ on $\partial \Omega \times [0, T]$ can be extracted easily.

In the next subsection, we discuss our choice of $\{\Psi_n\}_{n \geq 1}$ and the number $N$ in Section \ref{sec 2.1} and the truncation in \eqref{Fourier u}.

\subsection{A special orthonormal basis $\{\Psi_n\}_{n \geq 1}$ of $L^2(0, T)$ and the choice of the cut-off number $N$} \label{basis MK}

We will employ a special basis of $L^2(0, T)$. 
For each $n = 1, 2, \dots,$ set $\phi_n(t) = (t-T/2)^{n - 1}\exp(t-T/2)$. 
The set $\{\phi_n\}_{n = 1}^{\infty}$ is complete in $L^2(0, T).$
Applying the Gram-Schmidt orthonormalization process to this set, we obtain a basis of $L^2(0, T)$, named as $\{\Psi_n\}_{n = 1}^{\infty}$.
%We have the proposition
%\begin{proposition}[see \cite{Klibanov:jiip2017}]
%The basis $\{\Psi_n\}_{n = 1}^{\infty}$ satisfies the following properties:
%\begin{enumerate}
%	\item $\Psi_n'$ is not identically zero for all $n \geq 1$,
%	\item For all $m, n \geq 1$
%	\[
%		s_{mn} = \int_{0}^{T} \Psi_n'(t) \Psi_m(t) dt
%		= \left\{
%			\begin{array}{ll}
%				1 & \mbox{if } m = n,\\
%				0 & \mbox{if } n < m.
%			\end{array}
%		\right.
%%	\label{smn}
%	\]
%	As a result, for all integer $N > 1$, the matrix $S = (s_{mn})_{m, n = 1}^N$, is invertible.
%\end{enumerate}
%\label{prop MK}
%\end{proposition}
This basis was originally introduced to solve the electrical impedance tomography problem with partial data in \cite{Klibanov:jiip2017}.
Since then, this basis was widely used to solve a variety of inverse problems.
For instance, in \cite{Nguyens:jiip2019}, we employ this basis to solve an inverse source problem and a coefficient inverse problem for linear parabolic equations; 
in \cite{NguyenLiKlibanov:IPI2019}, this special basis was used to solve an inverse source problem for elliptic equations;
in \cite{KlibanovNguyen:ip2019}, we solve the problem of finding Radon inverse with incomplete data;
in \cite{KlibanovAlexeyNguyen:SISC2019}, we solve an inverse source problem for the full transport radiative equation.
The most related paper with the current one is \cite{LiNguyen:IPSE2019}, in which the second author and his collaborator employed this basis to recover the initial condition for linear parabolic equations.

We next discuss the choice of $N$ in \eqref{Fourier u}. 
Fix a positive integer $N_\x$.
On $\overline \Omega = [-R, R]^2,$ we arrange an $N_\x \times N_\x$ uniform grid
\[
	\mathcal{G} = \Big\{(x_i, y_j): x_i = -R + (i - 1)h, y_j = -R + (j - 1)h, 1 \leq i, j \leq N_\x\Big\}
\]
where $h = 2R/(N_\x - 1)$ is the step size. 
In our computations, we set $R_1 = 6$, $R = 1$, $T = 1.5$ and $N_\x = 80$.
To solve Problem \ref{ISP}, we need to compute the discrete values of the function $u$ on the grid $\mathcal{G}.$
%\subsection{Finding the cut-off number $N$}

\begin{figure}[h!]
	\begin{center}
		\subfloat[$N = 15$]{\includegraphics[width=0.3\textwidth]{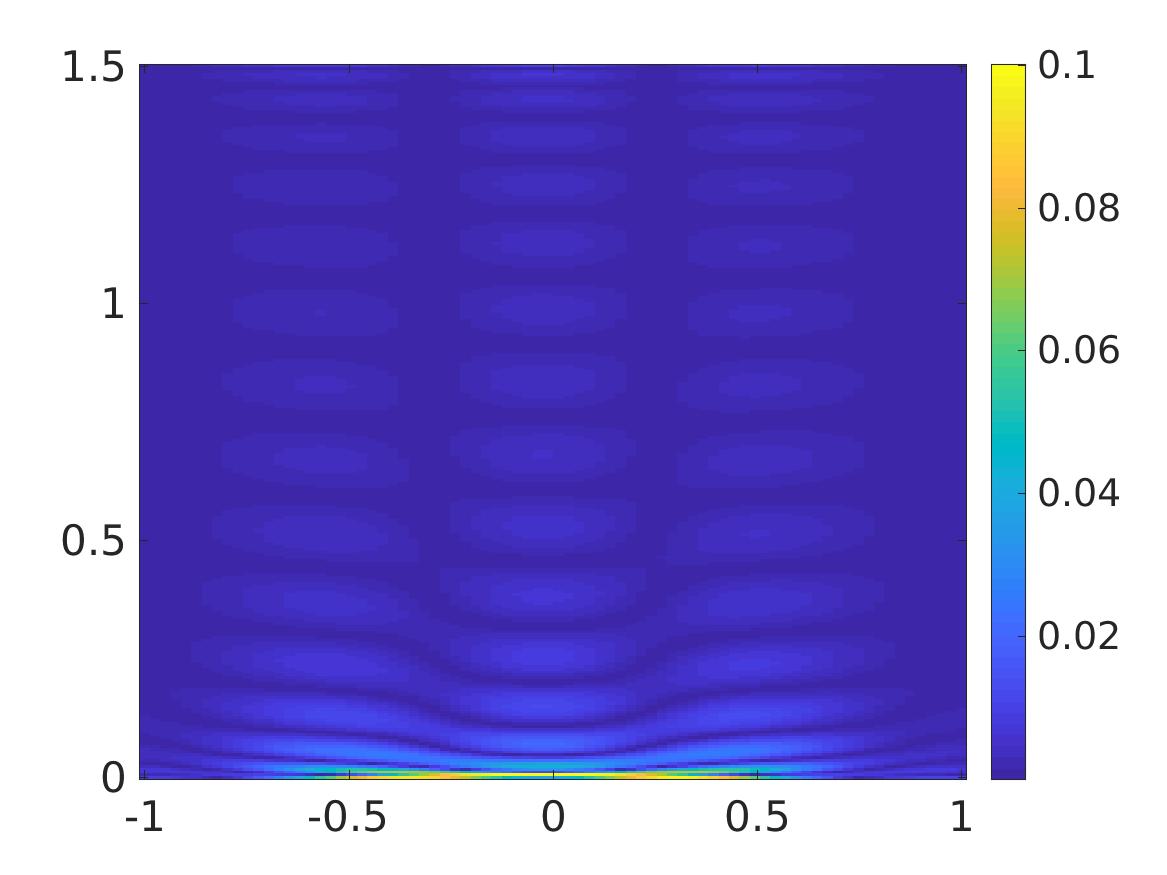}} \hfill
		\subfloat[$N = 25$]{\includegraphics[width=0.3\textwidth]{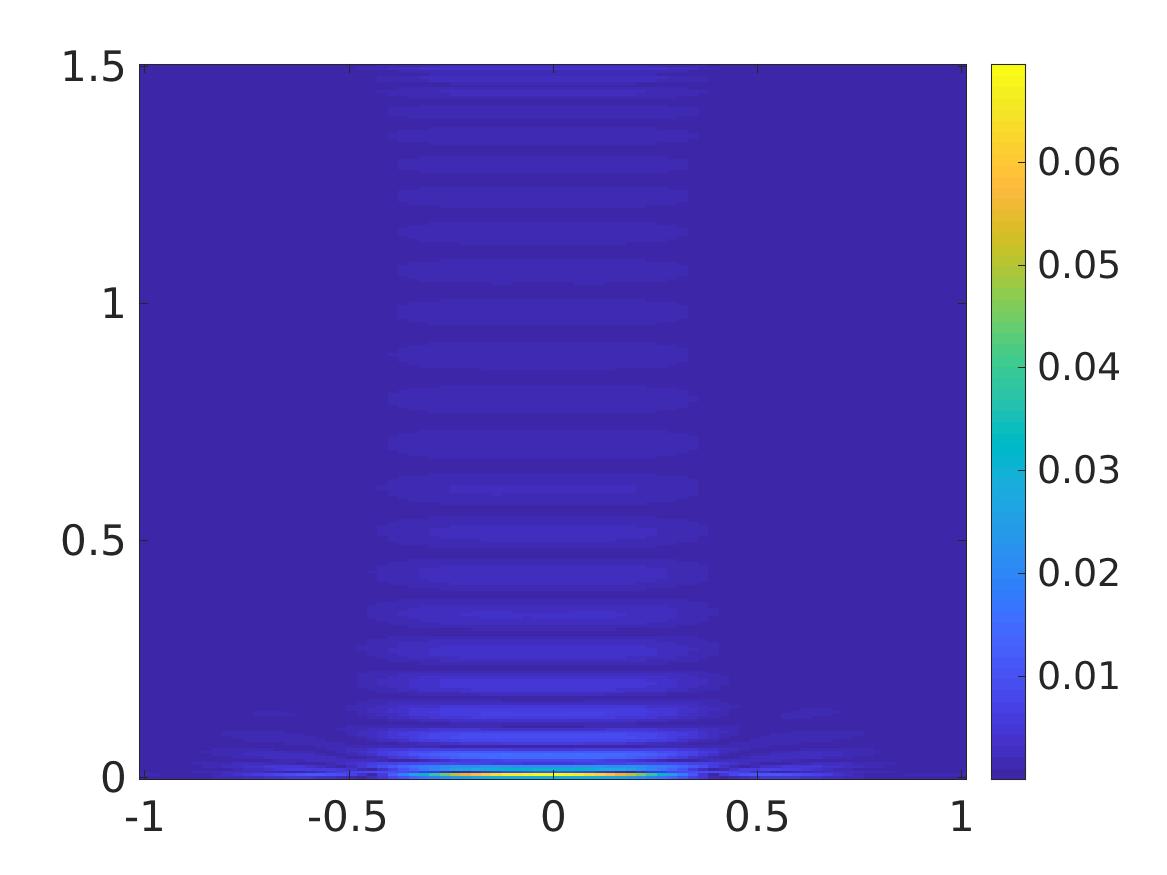}} \hfill
		\subfloat[$N = 35$]{\includegraphics[width=0.3\textwidth]{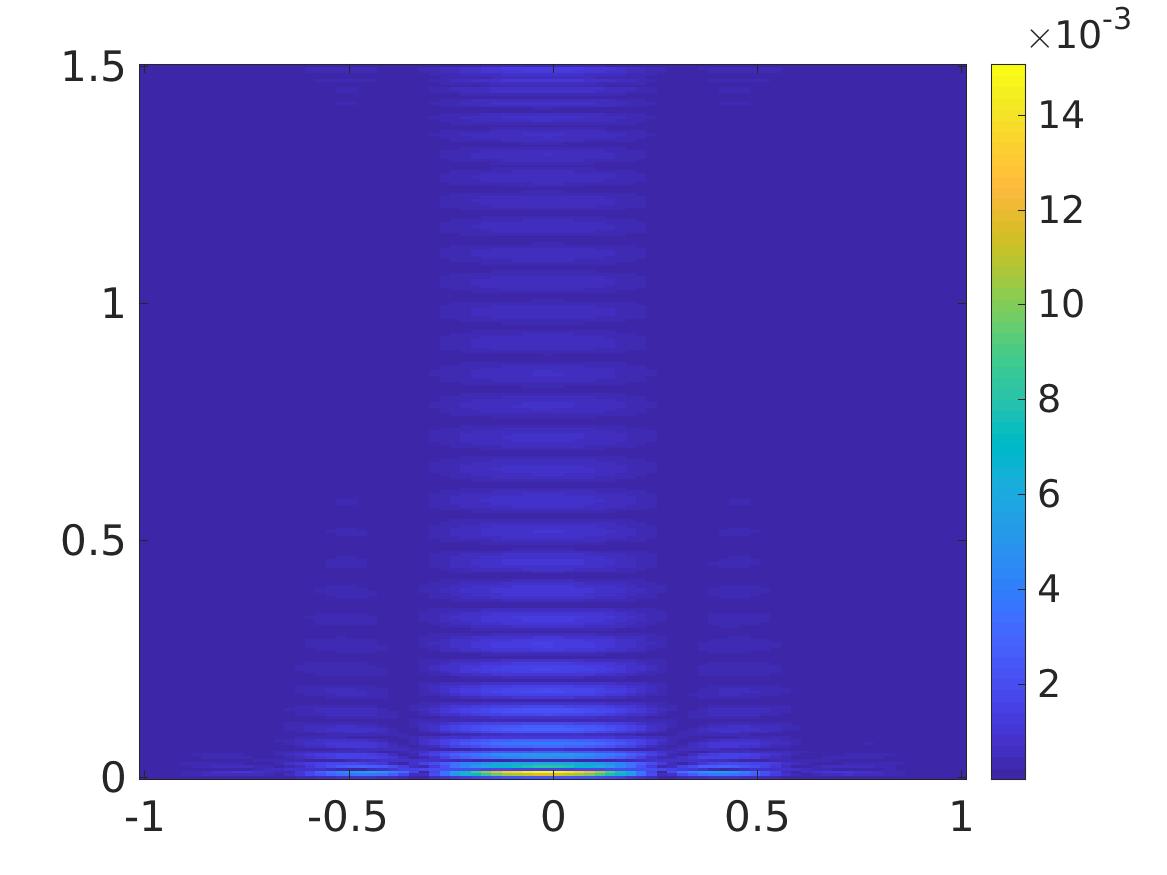}}
		
		\subfloat[$N = 15$]{\includegraphics[width=0.3\textwidth]{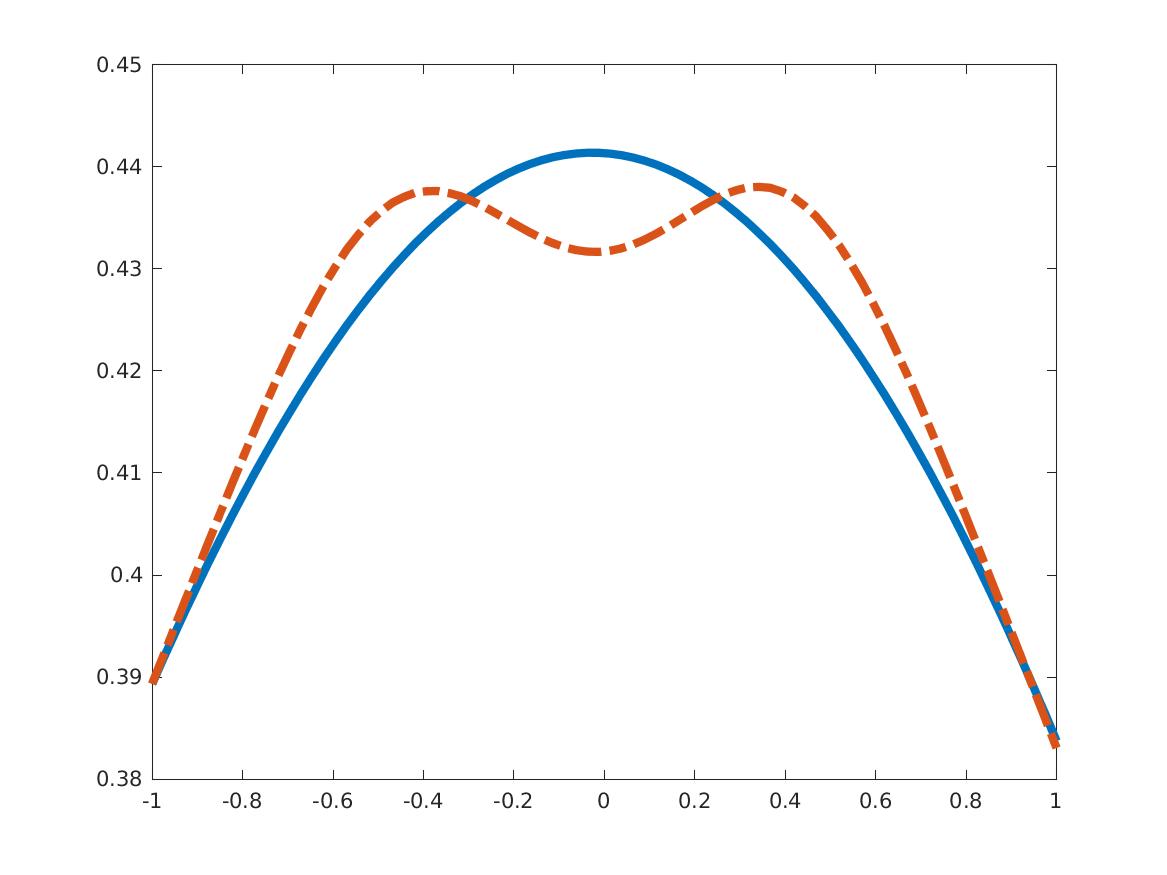}} \hfill
		\subfloat[$N = 25$]{\includegraphics[width=0.3\textwidth]{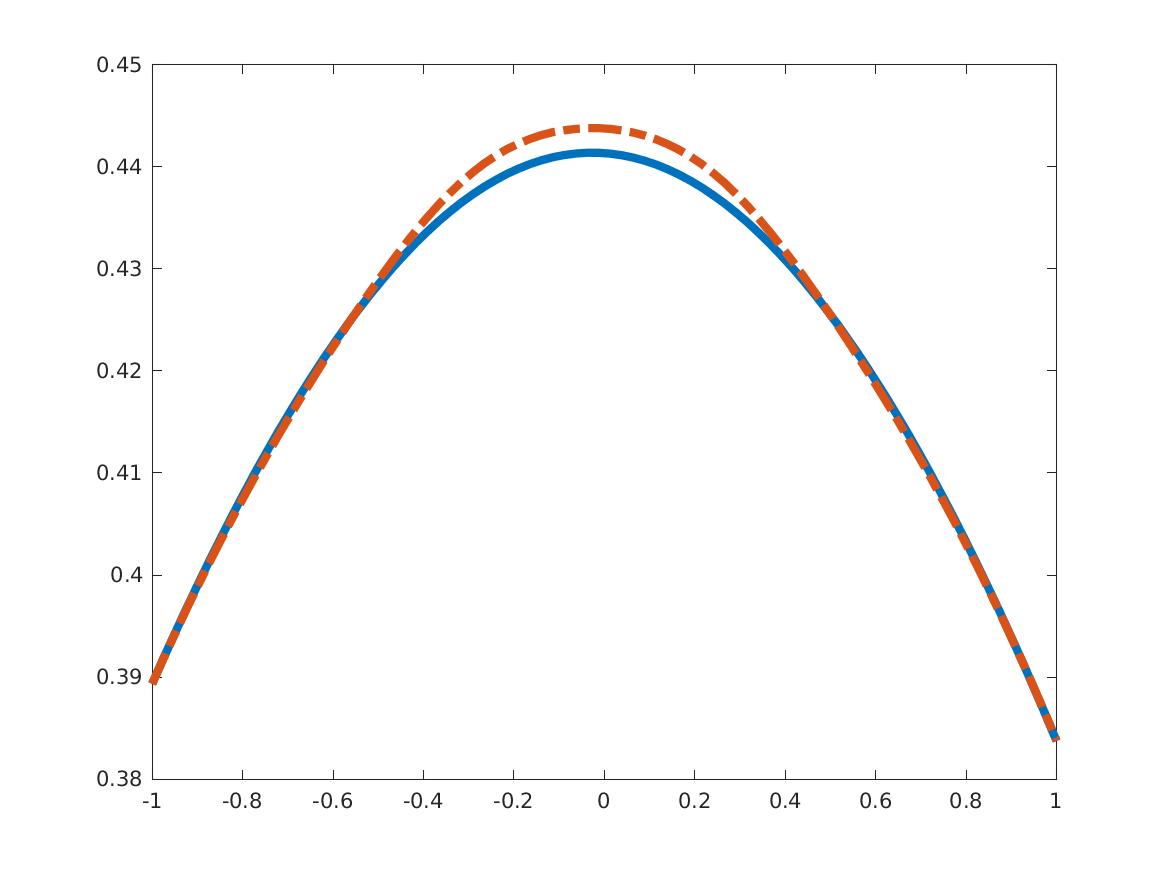}} \hfill
		\subfloat[$N = 35$]{\includegraphics[width=0.3\textwidth]{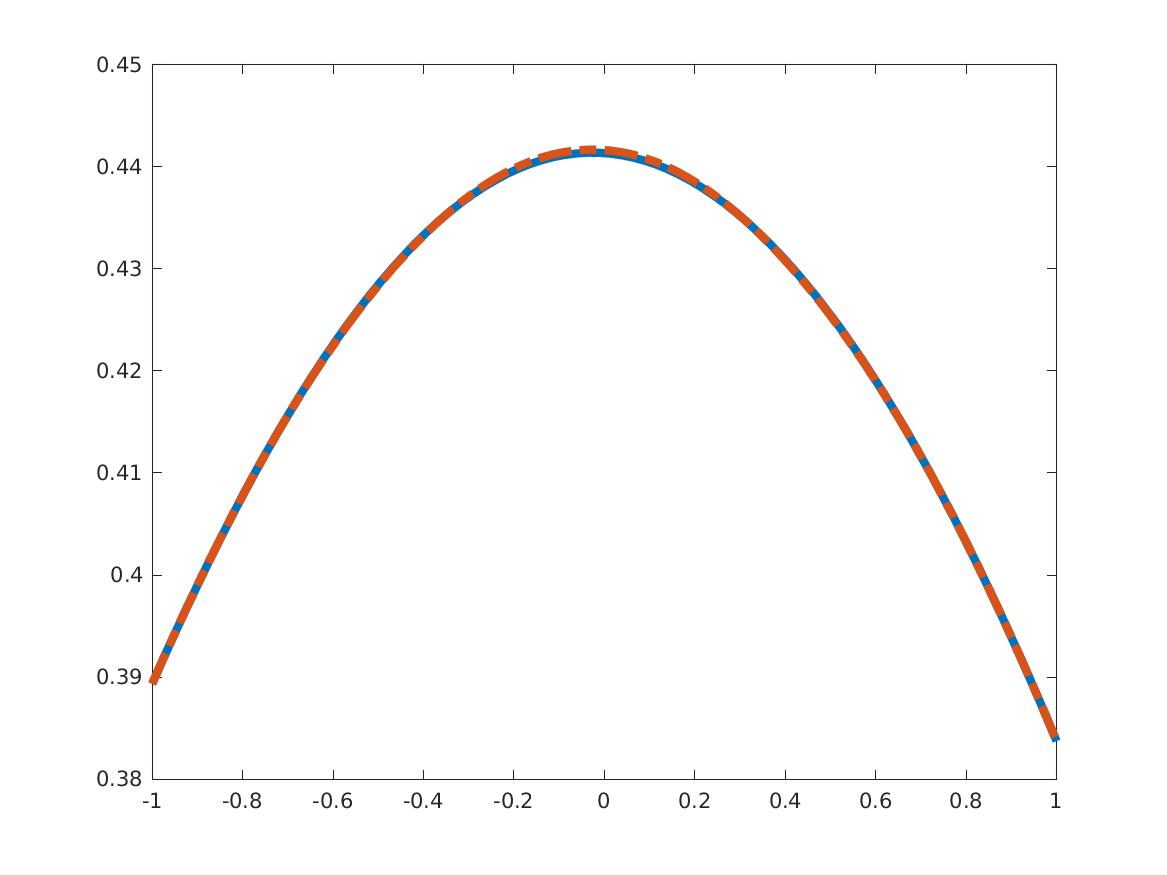}}
		\caption{\label{fig choose N} 
		The comparison of $f(x, R, t)$ and its partial Fourier sum $\sum_{n = 1}^N f_m(x, y = R, t)$ on $\{(x, y = R) \in \partial \Omega\}.$
		The first row displays the graphs of the absolute differences of $f(x, R, t)$ and $\sum_{n = 1}^N f_n(x,  R) \Psi_n(t)$. 
		The horizontal axis indicates $x$ and the vertical axis indicates $t$.
		It is evident that the bigger $N$, the smaller the difference is.		
			The second row shows the true data $f(x, y = R, T)$ (solid line) and its approximation $\sum_{n = 1}^N f_n(x, y =  R) \Psi_n(T)$ (dash--dot line).
			We observe that when $N = 35$, the two curves coincide.
%		It is evident that
%		the error $|f(x, y = R, t) - \sum_{n = 1}^N f_n(x,  y = R) \Psi(t)|$ is acceptably small when $N = 35.$
}
	\end{center}
\end{figure} 
The first step in our method is to find an appropriate cut off number $N$. 
We do so as follows.
Take the data on $\{(x, y = R) \in \partial \Omega\}$, which is the top part of $\partial \Omega$, $f(x, y = R, t) =  u_{\rm true}(x,  y = R, t)$ in Test 1 in Section \ref{sec num example}. 
Then, we compare the function $f(x, R, t)$ and the function $\sum_{n = 1}^N f_n(x, y= R) \Psi_n(t)$ where $f_n(x, y=R)$ is computed by (\ref{boundary conditions}). 
Choose $N$ such that the function \[e_N(\x, t) = \Big|f(x, y = R, t) - \sum_{n = 1}^N f_n(x,  y = R) \Psi_n(t)\Big|\] is small enough.
We use the same number $N$ for all numerical tests. 
In this paper, $N = 35,$ see Figure \ref{fig choose N} for an illustration.

\begin{remark}
	In our computations, when the cut-off number $N$ is $ 15$ or $25$, the  quality of the numerical results is poor. 
	When $N = 35$, we obtain good numerical results. 
	Increasing $N > 35$ does not improve the computed quality. 
	\label{rem N}
\end{remark}
\begin{remark}
In this numerical section, we choose the Carleman weight function  $e^{\lambda b^{-\beta} |\x - \x_0|^{\beta}}$ when defining $J^{(k)}$, $k \geq 0$, where $\lambda = 40$ and $\beta = 10$. The point $\x_0$ is $(0, 1.5)$ and $b = 5$. 
This and the condition $\lambda b^{-\beta}$ large conflict. 
However, in practice, the Carleman weight function with these values of $\lambda$ and $\beta$ already helps provide good numerical solutions to Problem \ref{ISP}. 
We numerically observe that the weight function blow-up when $\lambda b^{-\beta} \gg 1$, causing some unnecessary numerical difficulties.
\end{remark}

We next present the key step in the implementation of the inverse problem.
\subsection{Computing the vector-valued function $(u_m)_{m = 1}^N$}

Recall that $(u_m^{(0)}(x, y))_{m = 1}^N$ minimizes $J^{(0)}$ on $H$. Similarly to the argument in the first step of the proof of Theorem \ref{main thm}, for all $h \in H_0$, see the definition of $H_0$ in (\ref{H0}), by the variational principle, we have
\begin{equation}	
%\hspace{-2.5cm}			
			\sum_{m = 1}^N
			\Big\langle  e^{\lambda b^{-\beta} r^{\beta}(\x)}\Big[ \Delta u_m^{(0)} - c(\x)\sum_{n = 1}^N s_{mn} u_m^{(0)} \Big], 
	%		\nonumber
		%	\\		
	%		\hspace{4cm}	
			e^{\lambda b^{-\beta} r^{\beta}(\x)}\Big[ \Delta h_m - c(\x)\sum_{n = 1}^Ns_{mn} h_m\Big] \Big\rangle_{L^2(\Omega)} 
			= 0.
		\label{5.2}
		\end{equation}
For any $u \in H$,
we next associate the values of $u_m$ $\{u_m(x_i, y_j): 1 \leq m \leq N, 1 \leq i, j \leq N_\x\}$ with an $N_\x^2 N$ dimensional vector $\frak{u}_\frak{i}$
with 
\begin{equation}
	\frak{u}_\frak{i} = u_m(x_i, y_j)
	\label{u lineup}
\end{equation}
where
\begin{equation}
	\frak{i} = (i - 1)N_\x N + (j - 1) N + m
	\quad \mbox{for all }  1 \leq i, j \leq N_\x, 1 \leq m \leq N.
	\label{lineup}
\end{equation}
The range of the index $\frak{i}$ is $\{1, \dots, N_\x^2 N\}.$ 
The ``line-up" finite difference form of (\ref{5.2}) is 
\begin{equation}
	\langle  (\mathcal L - \mathcal S) \frak{u}^{(0)},  (\mathcal L - \mathcal S) \frak{h} \rangle = 0
	\label{5.5}
\end{equation}
where $\frak{W}_\lambda^2$, $\frak{u}^{(0)}$ and $\frak{h}$ are the line-up versions of  $(W_\lambda^2)_{m = 1}^N$, $(u_m^{0})_{m = 1}^N$ and $(h_m)_{m = 1}^N$ respectively.
Here, $\langle \cdot, \cdot \rangle$ is the classical Euclidian inner product. 
In (\ref{5.5}) 
\begin{enumerate}
\item the $N_\x^2 N \times N_\x^2 N$ matrix $\mathcal L$ is defined as  
\begin{enumerate}
	\item $(\mathcal L)_{\frak{i} \frak{i}} = -\frac{4 e^{\lambda b^{-\beta} r^{\beta}(x_i, y_j)}}{d_\x^2}$ for $\frak{i}$ as in (\ref{lineup}) for $2 \leq i, j \leq N_\x - 1$, $1 \leq m \leq N$;
	\item $(\mathcal L)_{\frak{i} \frak{j}} = \frac{e^{\lambda b^{-\beta} r^{\beta}(x_i, y_j)}}{d_\x^2}$ for $\frak{j} =  (i \pm 1 - 1)N_\x N + (j - 1) N + m$ and $\frak{j} =  (i  - 1)N_\x N + (j \pm 1 - 1) N + m$; for $2 \leq i, j \leq N_\x - 1$, $1 \leq m \leq N$;
	\item the other entries are $0$.
\end{enumerate}
\item the $N_\x^2 N \times N_\x^2 N$ matrix $\mathcal S$ is defined  as
		 $(\mathcal S)_{\frak{i}\frak{j}} =  e^{\lambda b^{-\beta} r^{\beta}(x_i, y_j)}c(x_i, y_j) s_{m n}$ for $\frak{i}$ as in (\ref{lineup}) and $\frak{j} =  (i  - 1)N_\x N + (j \pm 1 - 1) N + n$ for $2 \leq i, j \leq N_\x - 1$, $1 \leq m, n \leq N$.	 
		 The other entries are $0$.
\end{enumerate}
%It follows from (\ref{5.5}) that
%\[
%	\langle \frak{W}_\lambda^2 (\mathcal L - \mathcal S)^T(\mathcal L - \mathcal S) \frak{u}^{(0)},   \frak{h} \rangle = 0
%\]
%for all $\frak{h}$. 
%Hence, 
%\begin{equation}
%	(\mathcal L - \mathcal S)^T(\mathcal L - \mathcal S) \frak{u}^{(0)} = 0.
%	\label{5.6}
%\end{equation}
On the other hand, since $(u^{0}_m)_{m = 1}^N$ satisfies the boundary constraints (\ref{boundary conditions}), we have
\begin{equation}
	\mathcal D \frak{u}^{(0)} = \frak{f} \quad
	\mbox{and }
	\quad \mathcal N \frak{u}^{(0)} = \frak{g}
	\label{5.7}
\end{equation}
where
\begin{enumerate}
\item The $N_\x^2 N \times N_\x^2 N$ matrix $\mathcal D$ is  defined as $\mathcal D_{\frak{i}\frak{i}} = 1$ for $\frak{i}$ as in (\ref{lineup}), $i \in \{1, N_\x\},$ $1 \leq j \leq N_\x$ or $2 \leq i \leq N_\x - 1$, $j \in \{1, N_\x\}.$
 The other entries are $0$.
 \item The $N_\x^2 N \times N_\x^2 N$ matrix $\mathcal N$ is  defined as 
 	\begin{enumerate}
 		\item $\mathcal N_{\frak{i}\frak{i}} = \frac{1}{d_\x}$ for $\frak{i}$ as in (\ref{lineup}), $i \in \{1, N_\x\},$ $1 \leq j \leq N_\x$ or $2 \leq i \leq N_\x - 1$, $j \in \{1, N_\x\}$, $1 \leq m \leq N;$
		\item $\mathcal N_{\frak{i}\frak{j}} = -\frac{1}{d_\x}$ for $\frak{i}$ as in (\ref{lineup}) and $\frak{j} = (i + 1- 1)N_\x N + (j - 1) N + m$, $i = 1,$ $1 \leq j \leq N_\x$,   $1 \leq m \leq N;$
		\item $\mathcal N_{\frak{i}\frak{j}} = -\frac{1}{d_\x}$ for $\frak{i}$ as in (\ref{lineup}) and $\frak{j} = (i - 1- 1)N_\x N + (j - 1) N + m$, $i = N_\x,$ $1 \leq j \leq N_\x$,   $1 \leq m \leq N;$
		\item $\mathcal N_{\frak{i}\frak{j}} = -\frac{1}{d_\x}$ for $\frak{i}$ as in (\ref{lineup}) and $\frak{j} = (i - 1)N_\x N + (j + 1 - 1) N + m$,  $1 \leq i \leq N_\x$, $j = 1$,   $1 \leq m \leq N;$
		\item $\mathcal N_{\frak{i}\frak{j}} = -\frac{1}{d_\x}$ for $\frak{i}$ as in (\ref{lineup}) and $\frak{j} = (i - 1)N_\x N + (j - 1 - 1) N + m$,  $2 \leq i \leq N_\x -1 $,  $j = N_\x,$ $1 \leq m \leq N;$
		\item The other entries are $0$.
 	\end{enumerate}
	\item The $N_\x^2 N$ dimensional vector $\frak{f}$ is defined as $\frak{f}_{\frak{i}} = f_m(x_i, y_j)$ for $\frak{i}$ as in (\ref{lineup}), $i \in \{1, N_\x\},$ $1 \leq j \leq N_\x$, $1 \leq m \leq N$ or $2 \leq i \leq N_\x - 1$, $j \in \{1, N_\x\}.$
	\item The $N_\x^2 N$ dimensional vector $\frak{g}$ is defined as $\frak{g}_{\frak{i}} = g_m(x_i, y_j)$ for $\frak{i}$ as in (\ref{lineup}), $i \in \{1, N_\x\},$ $1 \leq j \leq N_\x$, $1 \leq m \leq N$ or $2 \leq i \leq N_\x - 1$, $j \in \{1, N_\x\}.$
\end{enumerate}
Solving (\ref{5.5})--(\ref{5.7}) by the least square method with the command ``lsqlin" built in Matlab, we obtain the vector $\frak{u}^{(0)}$ and hence the initial solution $(u_m^{(0)}(x_i, y_j))_{m = 1}^N$ for $1 \leq i, j \leq N_\x$.

\begin{remark}
	In computation, defining the matrices above is ineffective due to their large size, $N_\x^2 N \times N_\x^2 N$ where $N_\x = 80$ and $N = 35$.	
	We note that most of those matrices' entries are 0.
	So, instead of defining dense matrices, we use the invention of sparse matrices. 
	Moreover, using sparse matrices significantly reduces the computational time.
\end{remark}

We next compute the vector valued function $(u_m^{(k)})_{m = 1}^N$, $k \geq 1$, assuming by induction that $(u_m^{(k-1)})_{m = 1}^N$ is known.
Applying a very similar argument when deriving (\ref{5.5})--(\ref{5.7}), the vector $\frak{u}^{(k)}$ the line up version of $(u_m^{(k)}(x_i, y_j))_{m = 1}^N$ with $1 \leq i, j \leq N_\x$ satisfies the equations
\begin{equation}
	(\mathcal L - \mathcal S)^T(\mathcal L - \mathcal S) \frak{u}^{(k)} = -(\mathcal L - \mathcal S)^T\frak{q}^{(k-1)}.
	\label{5.8}
\end{equation}
and
\begin{equation}
	\mathcal D \frak{u}^{(k)} = \frak{f} \quad
	\mbox{and }
	\quad \mathcal N \frak{u}^{(k)} = \frak{g}
	\label{5.9}
\end{equation}
where $\frak{q}^{(k-1)}$ is the line up version of $(q_m(u_1^{(k - 1)}(x, y), \dots, u_N^{(k - 1)}(x, y)))_{m = 1}^N$.
To find $\frak{u}^{(k)}$, we solve (\ref{5.8})--(\ref{5.9}) by the least square method with the command ``lsqlin" of Matlab. 
The value of the function $(u_m(x_i, y_j))_{m = 1}^N$ follows.
We next find $u(x, y, t)$ via (\ref{Fourier u}). 
The desired solution to Problem \ref{ISP} $p(x, y)$ is set to be $u(x, y, 0)$.

\begin{remark}
	In theory, we need to apply the cut-off function $P$, see (\ref{cutoff function}). This is only for our convenience to prove Theorem \ref{main thm}.
	However, in computation, we can obtain good numerical results without applying the cut-off technique. 
	This can be explained by setting $M$ sufficiently large.
\end{remark}

We summarize the procedure to find $p$ in Algorithm \ref{alg}.

\begin{algorithm}[h!]
\caption{\label{alg}The procedure to solve Problem \ref{ISP}}
	\begin{algorithmic}[1]
	\State\, 
	Compute $\{\Psi_n\}_{n = 1}^N$ as in Section \ref{basis MK}.
	Choose $N = 35$, see Figure \ref{fig choose N} and Remark \ref{rem N}.
	\State\, Compute matrices $\mathcal{L}, \mathcal S, \mathcal D$ and $\mathcal N.$ Find the line up versions $\frak{f}$ and $\frak{g}$ of the data $f_m(x_i, y_j)$ and $g_m(x_i, y_j)$ for $(x_i, y_j) \in \mathcal G \cap \partial \Omega$, $1 \leq m \leq N$.
	\State\, \label{Step 3} Solve (\ref{5.5})--(\ref{5.7}) by the least square method. The solution is denoted by $\frak{u}^{(0)}$.
	Compute $u^{(0)}_m(x_i, y_j)$, $1 \leq i, j \leq N_\x$, $1 \leq m \leq N$ using
	$
		u^{(0)}_m(x_i, y_j) = (\frak{u}^{(0)})_{\frak{i}}
	$ with $\frak{i}$ as in (\ref{lineup}).	
	\State\, Set the initial solution $p^{(0)} = \sum_{n = 1}^N u^{(0)}_n(x_i, y_j)\Psi_n(0).$
	\For{$k = 1$ to $5$}
		\State Find $\frak{q}^{(k-1)}$, the line up version of $q(P(u_1^{(k-1)}(x_i, y_j)), \dots, u_N^{(k-1)}(x_i, y_j)))$, $1 \leq i, j \leq N_\x$, $1 \leq m \leq N$ in the same manner of (\ref{u lineup}) and (\ref{lineup}).
		\State\, Solve (\ref{5.8})--(\ref{5.9}) by the least square method. The solution is denoted by $\frak{u}^{(k)}$.
	Compute $u^{(k)}_m(x_i, y_j)$, $1 \leq i, j \leq N_\x$, $1 \leq m \leq N$ using
	$
		u^{(k)}_m(x_i, y_j) = (\frak{u}^{(k)})_{\frak{i}}
	$ with $\frak{i}$ as in (\ref{lineup}).	
	\State\, \label{Step 8} Set the initial solution $p^{(k)} = \sum_{n = 1}^N u^{(k)}_n(x_i, y_j)\Psi_n(0).$
	\State\, \label{error estimate} Define the recursive error at step $k$ as $\|p^{(k)} - p^{(k - 1)}\|_{L^\infty}(\Omega)$.
	\EndFor 
\end{algorithmic}
\end{algorithm}

\begin{remark}
	We numerically observe that $\|p^{(5)} - p^{(4)}\|_\infty$ is sufficiently small in all tests in Section \ref{sec num example}; 
	i.e., our iterative scheme converges fast. 
	Iterating the loop in Algorithm \ref{alg} five (5) times is enough to obtain good numerical results.
	Therefore, we stop the iterative process when $k = 5.$
	%Theorem \ref{main thm} is numerically confirmed.
\end{remark}

\subsection{Numerical examples} \label{sec num example}

In this section, we show four (4) numerical results.

\noindent{\bf Test 1.}
The true source function is given by
\[
	p_{\rm true} = \left\{
		\begin{array}{ll}
			8 & x^2 + (y - 0.3)^2 < 0.45^2,\\
			0 & \mbox{otherwise.}
		\end{array}
	\right.
\]
The  nonlinearity $q$ is given by 
\[
	q(s) = s(1 - s) \quad s \in \R.
\]
In this case, the parabolic equation in (\ref{main eqn}) is the Fisher equation.
The true and computed source functions $p$ are displayed in Figure \ref{example 1}.
It appears in the graph of this source function a big inclusion with contrast $8$.

\begin{figure}[h!]
	\begin{flushleft}
		\subfloat[]{\includegraphics[width = 0.3\textwidth]{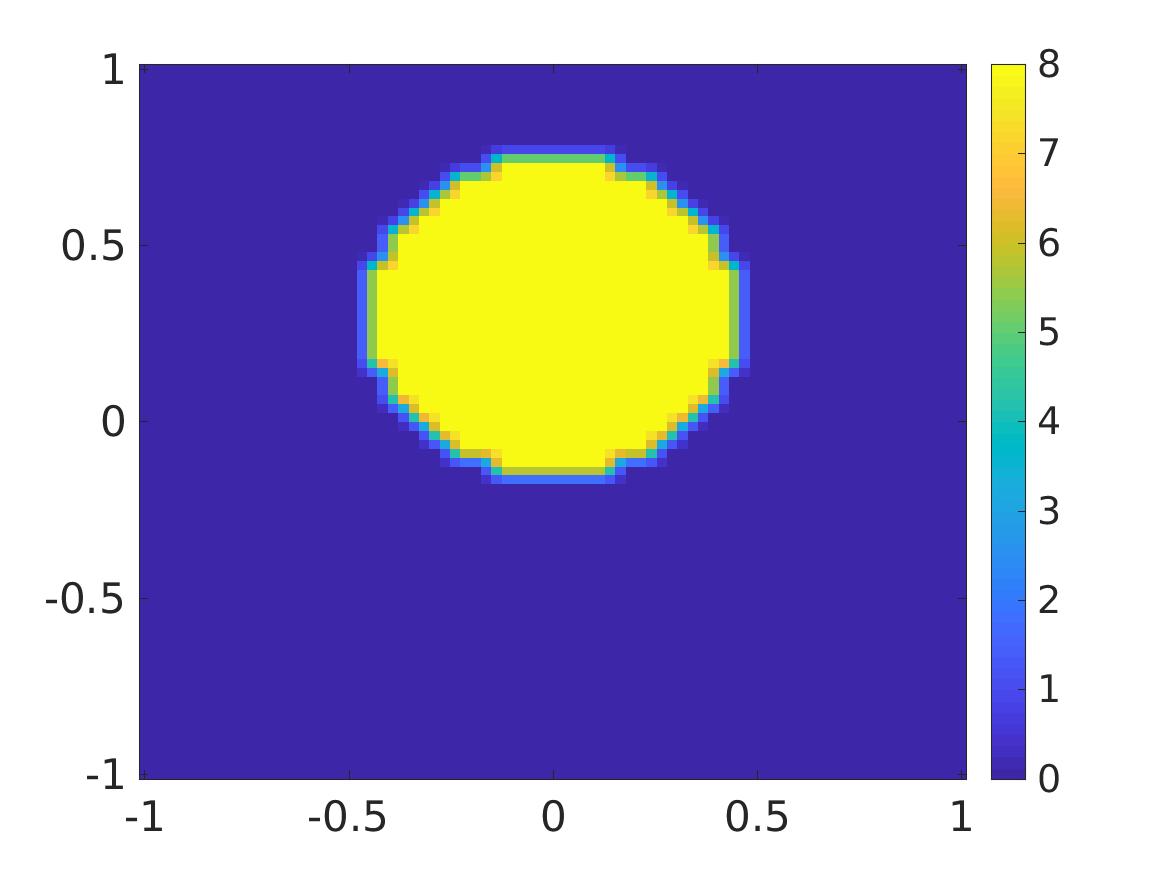}}
		\quad
		\subfloat[\label{m1 init20}]{\includegraphics[width = 0.3\textwidth]{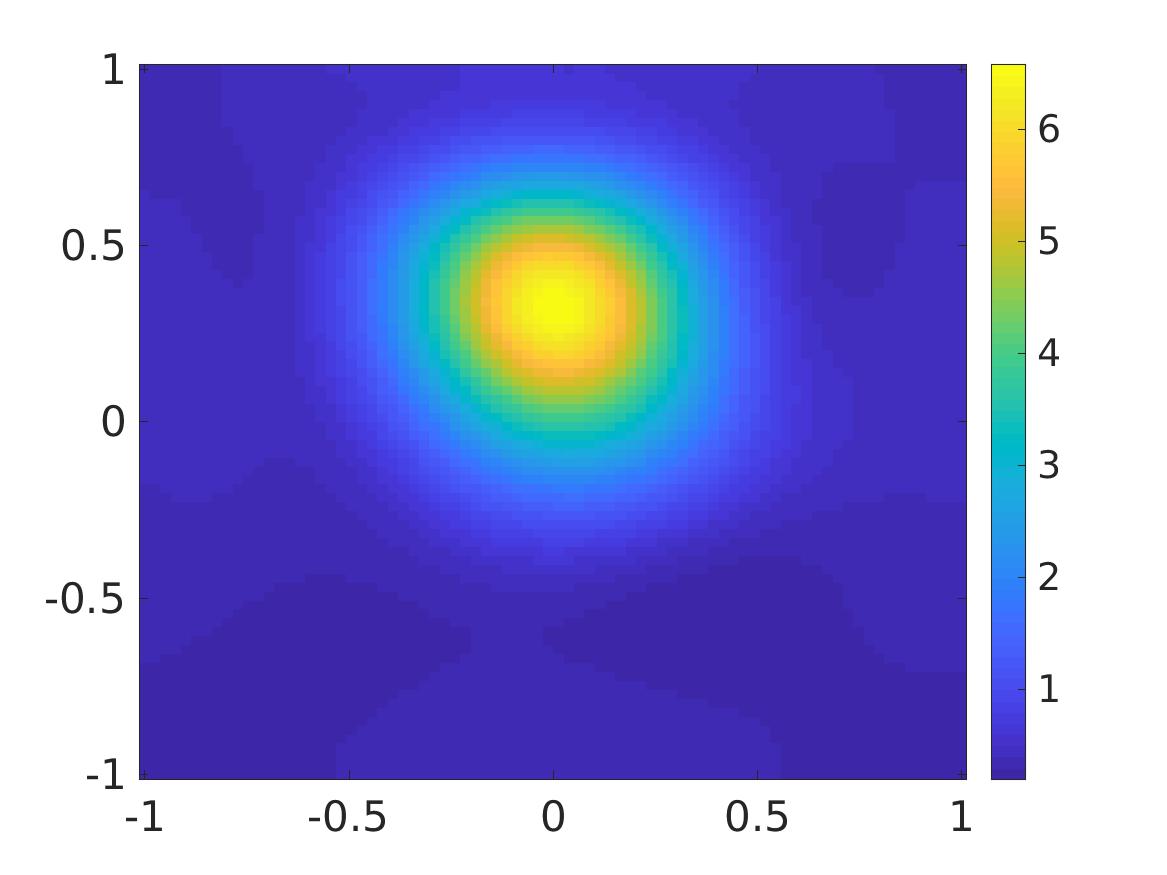}} ~ \quad

		\subfloat[\label{m1 p comp}]{\includegraphics[width = 0.3\textwidth]{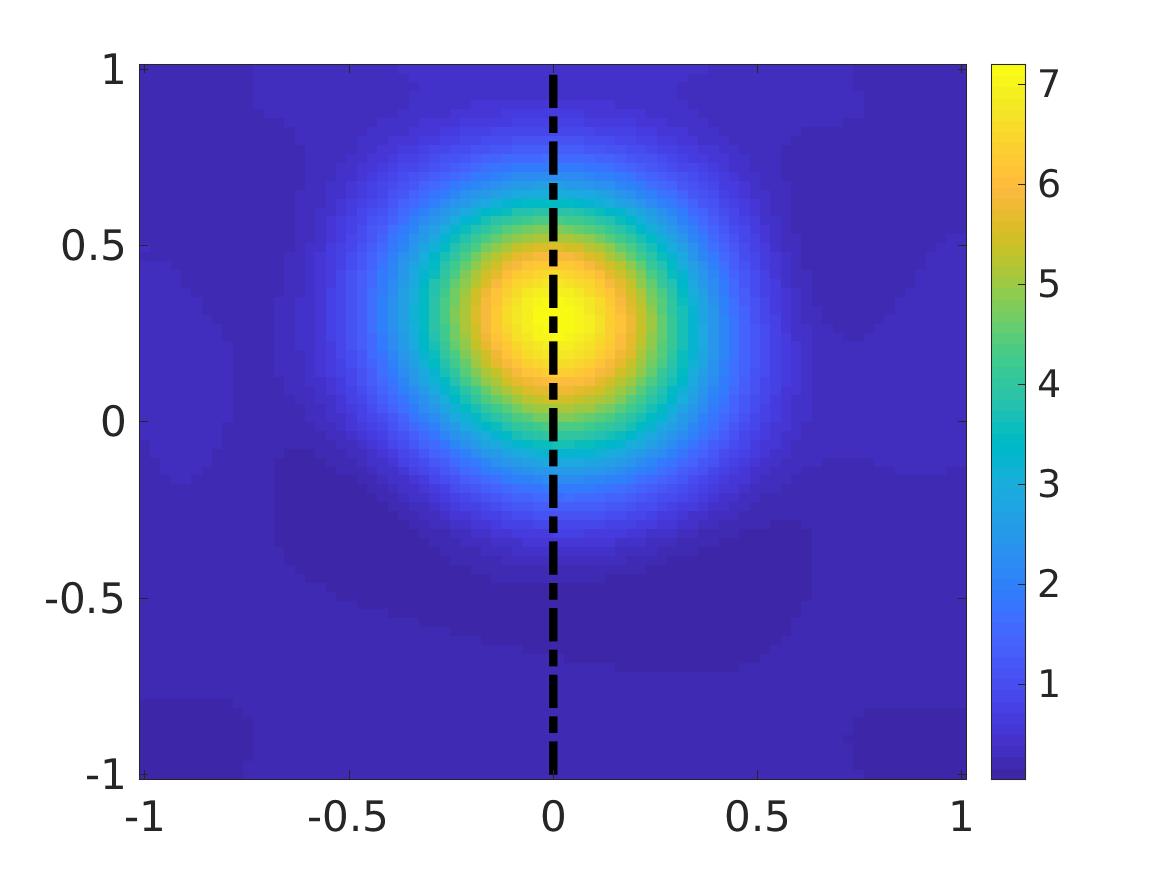}}\quad		
		\subfloat[\label{m1 cross}]{\includegraphics[width = 0.3\textwidth]{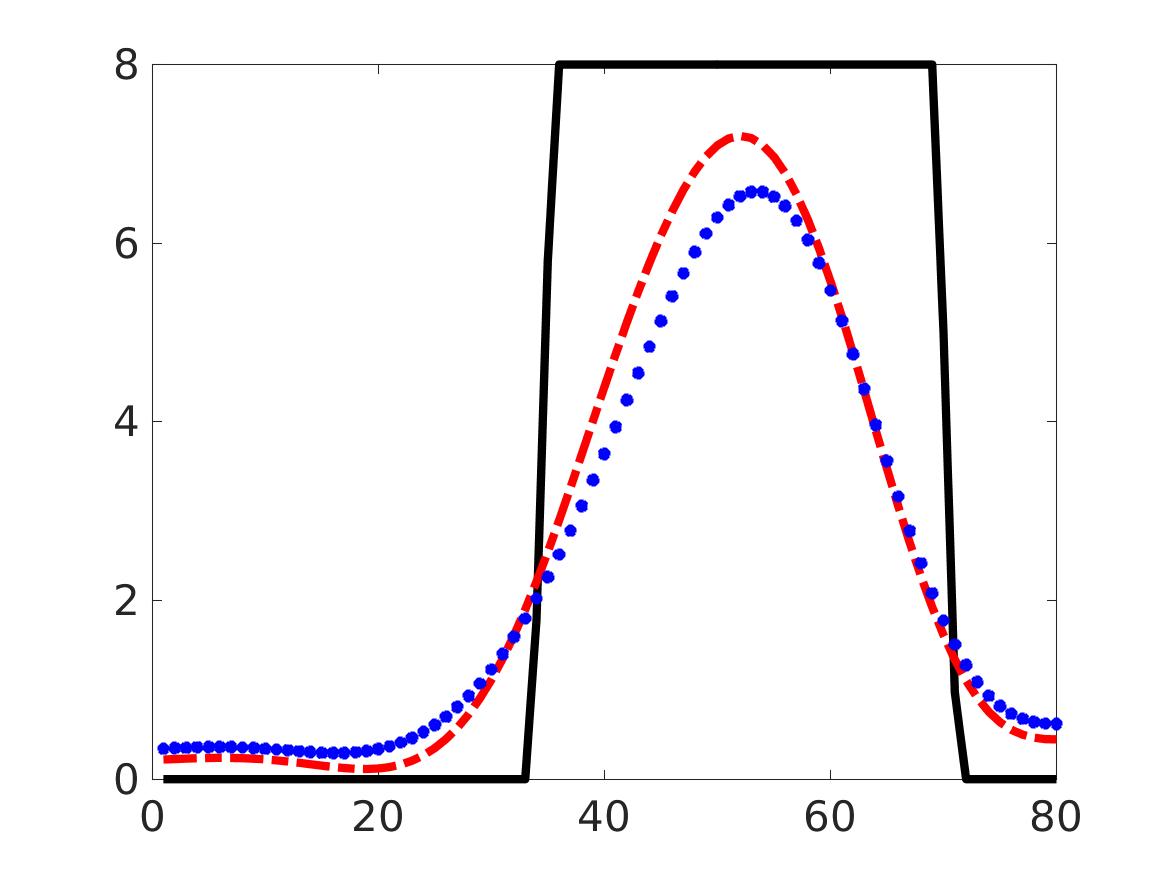}}
		\quad
		\subfloat[\label{fig m1 error}]{\includegraphics[width = 0.3\textwidth]{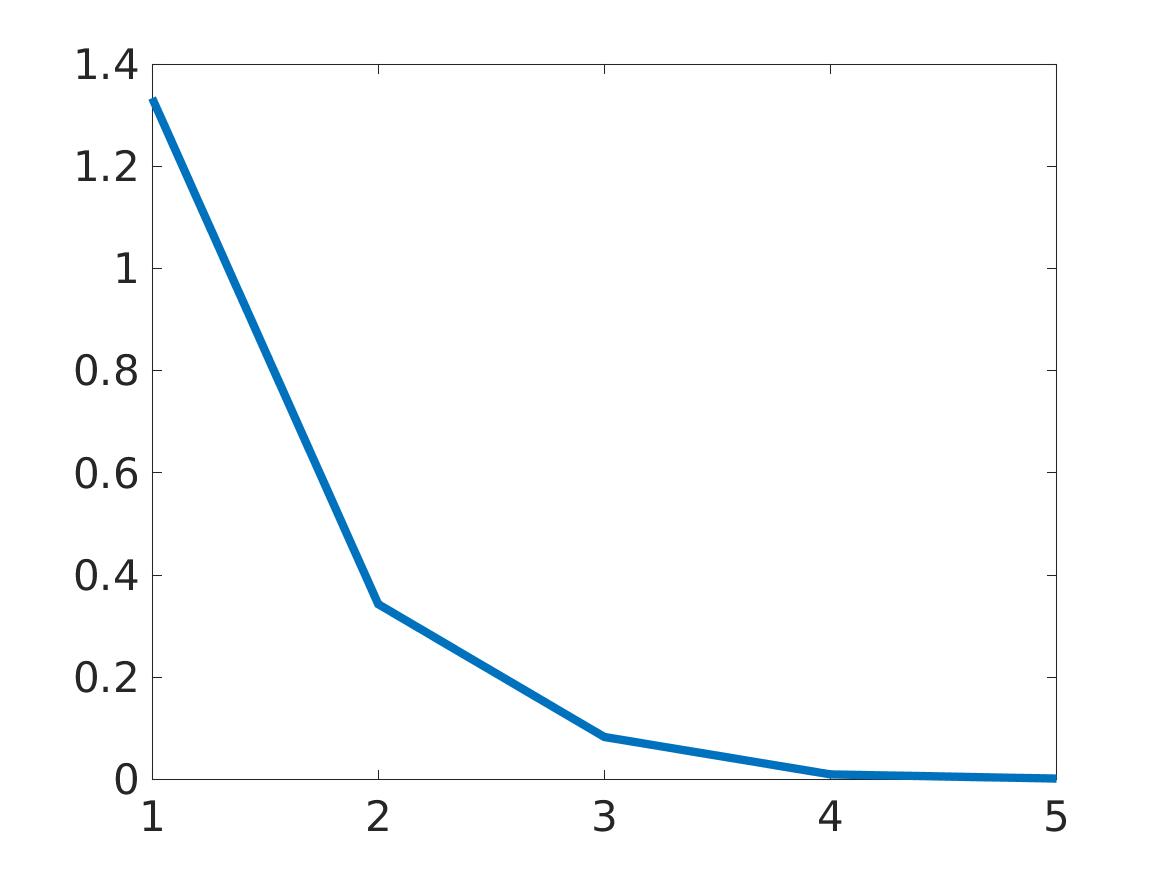}}
	\caption{\label{example 1} Test 1. The reconstruction of the source function. (a) The function $p_{\rm true}$
	(b) The initial solution $p^{(0)}$ obtained by Step \ref{Step 3} in Algorithm \ref{alg}.
	(c) The function $p^{(5)}$ obtained by Step \ref{Step 8} in Algorithm \ref{alg}.
	(d) The true (solid), the initial source function (dot) in (b) and computed source function (dash-dot) on the vertical line  in (c).
	(e) The curve $\|p^{(k)} - p^{(k - 1)}\|_{L^{\infty}(\Omega)},$ $k = 1, \dots, 5.$
	The noise level of the data in this test is $20\%$.
	}
	\end{flushleft}
\end{figure} 

Our method to find the initial solution works very well in this case.
One can see in Figure \ref{m1 init20} that by solving the system (\ref{5.5})--(\ref{5.7}), we obtain the initial solution that clearly indicates the position of the inclusion.
The value of the reconstructed function inside the inclusion is somewhat acceptable and will improve after several iterations, see Figure \ref{m1 cross}.
The reconstructed function $p_{\rm comp} = p^{(5)}$ is a good approximation of the true function $p_{\rm true}$, see Figures \ref{m1 p comp} and \ref{m1 cross}. 
It is evident from Figure \ref{fig m1 error} that our method converges fast.
The reconstructed maximal value inside the inclusion is 7.202 (relative error 9.98\%). 

\noindent{\bf Test 2.} We test the case of multiple inclusions, each of which has a different value.
 The true source function $p_{\rm true}$ is given by
\[
	p_{\rm true}(x, y) = \left\{
		\begin{array}{ll}
			12 & (x - 0.5)^2 + (y - 0.5)^2 < 0.35^2,\\
			10& (x + 0.5)^2 + (y + 0.5)^2 < 0.35^2,\\
			14& (x - 0.5)^2 + (y + 0.5)^2 < 0.35^2,\\
			9& (x + 0.5)^2 + (y - 0.5)^2 < 0.35^2,\\
			0 &\mbox{otherwise.}
		\end{array}
	\right.
\]
In this test, the nonlinearity $q$ is given by
\[
	q(s) = -s(1 - \sqrt{|s|}) \quad s \in \R.
\]
The true and computed source functions $p$ are displayed in Figure \ref{example 2}.
\begin{figure}[h!]
	\begin{flushleft}
		\subfloat[]{\includegraphics[width = 0.3\textwidth]{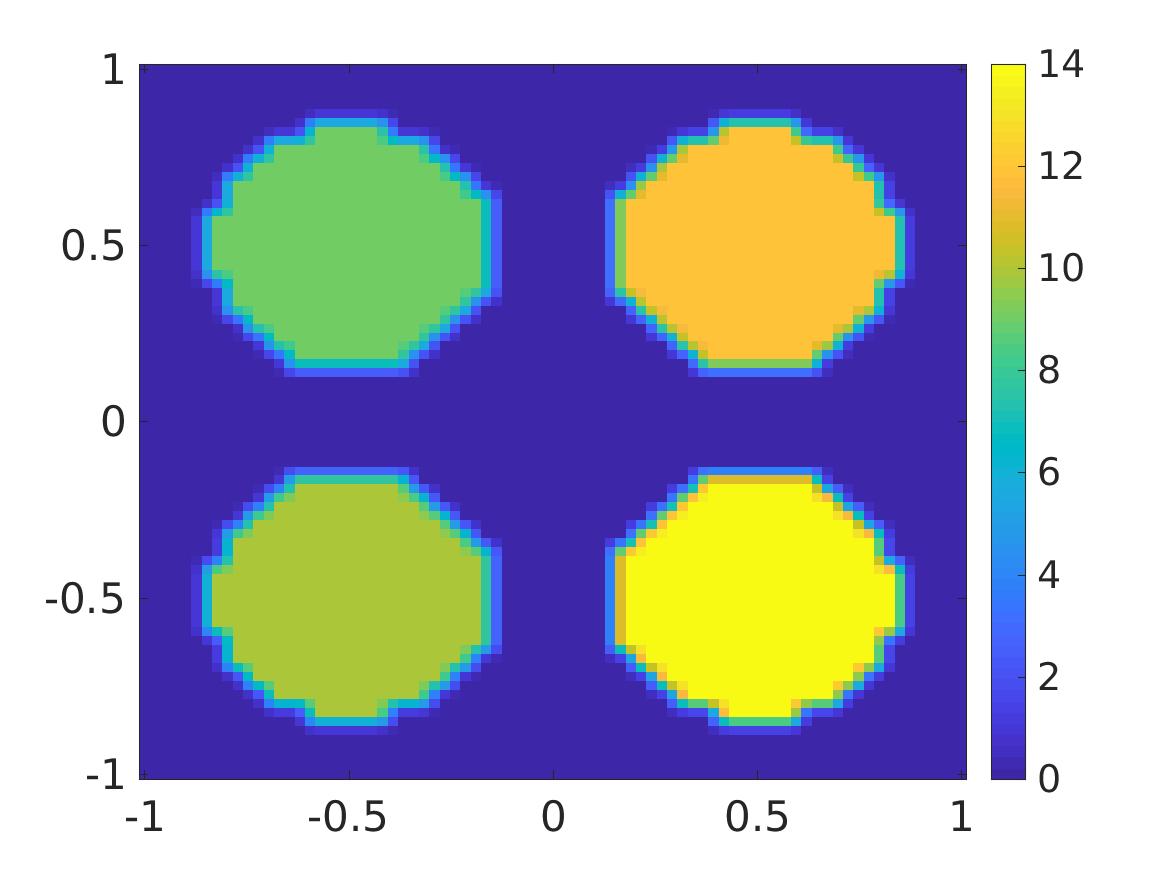}}
		\quad
		\subfloat[\label{m2 init20}]{\includegraphics[width = 0.3\textwidth]{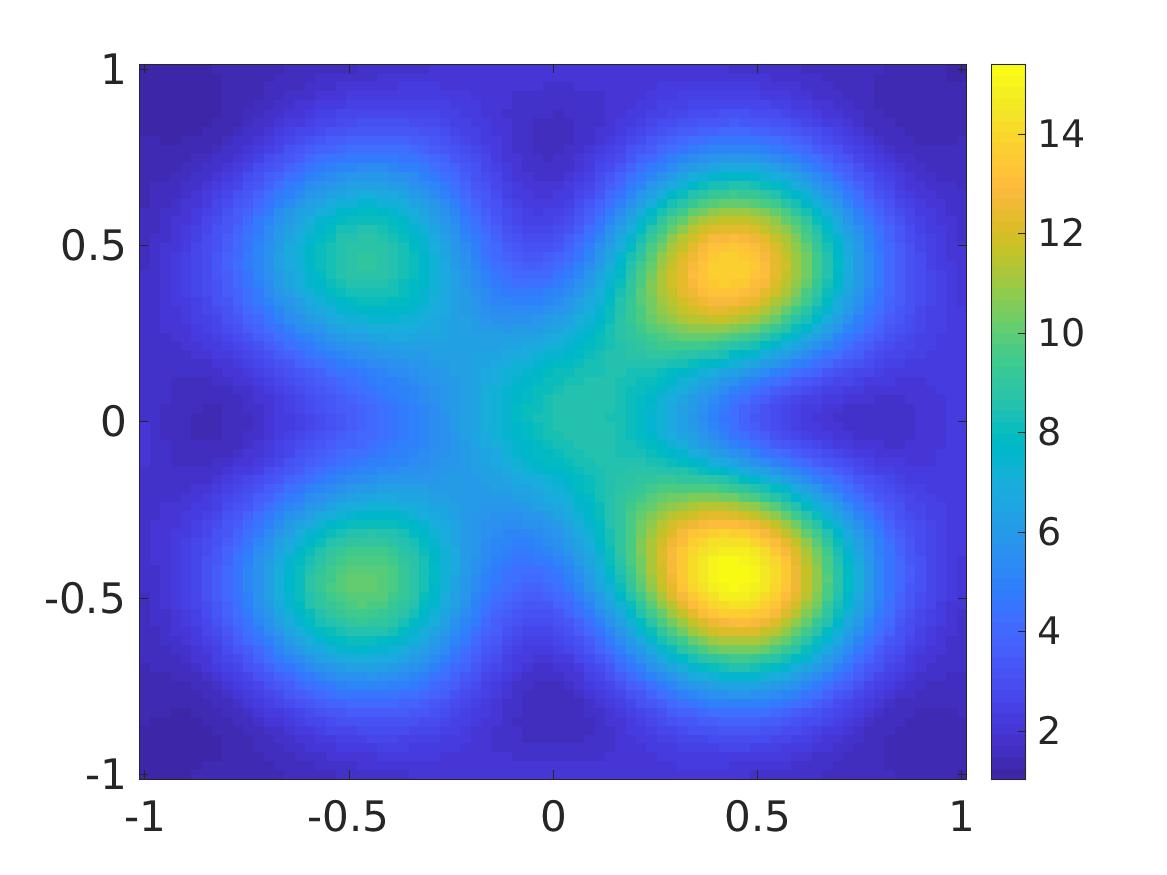}} ~ \quad

		\subfloat[\label{m2 p comp}]{\includegraphics[width = 0.3\textwidth]{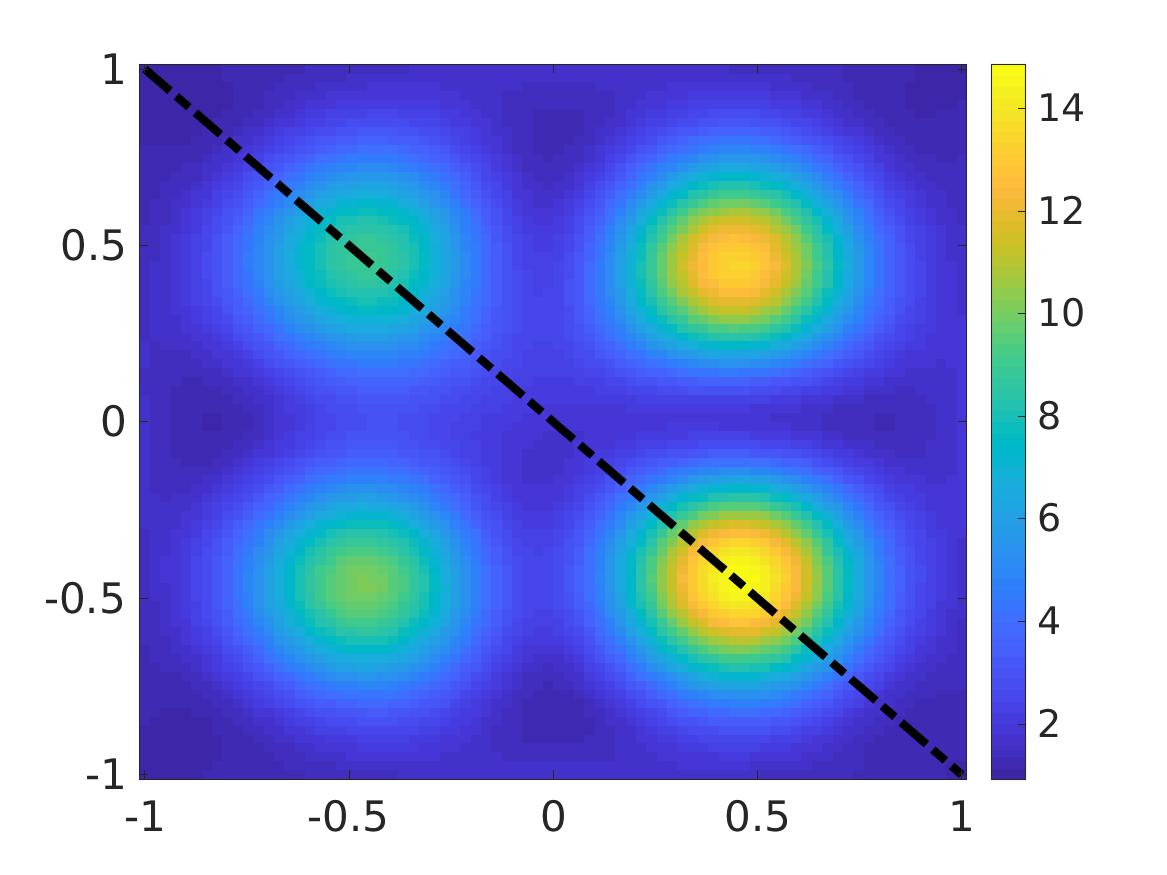}}\quad		
		\subfloat[\label{m2 cross}]{\includegraphics[width = 0.3\textwidth]{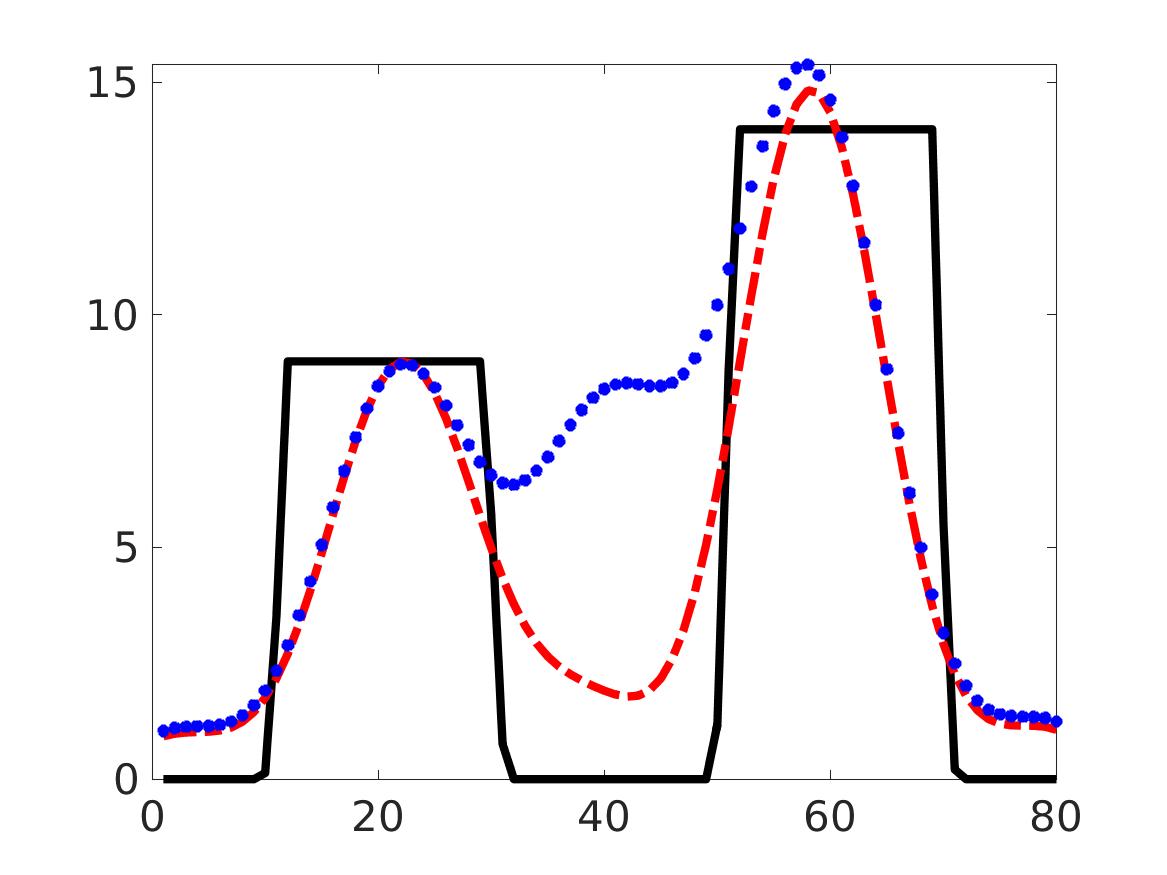}}
		\quad
		\subfloat[\label{fig m2 error}]{\includegraphics[width = 0.3\textwidth]{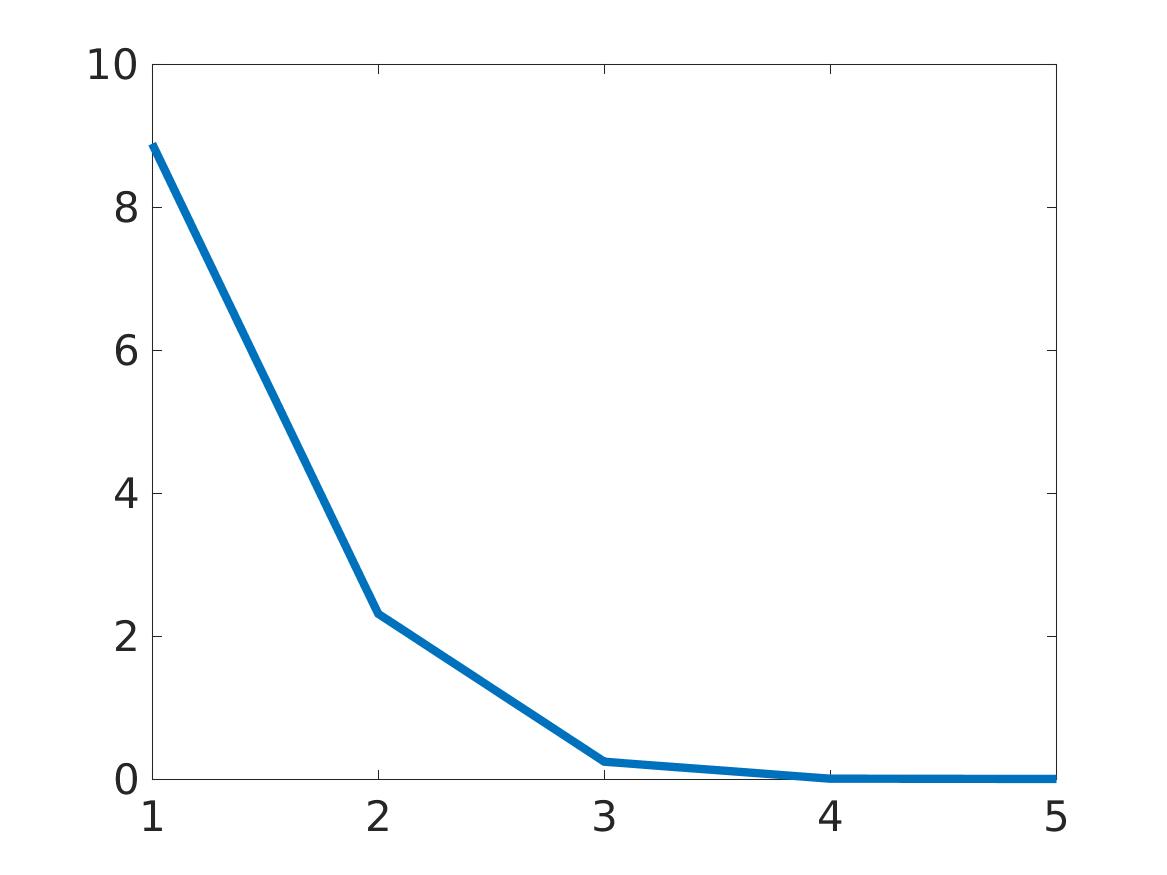}}
	\caption{\label{example 2} Test 2. The reconstruction of the source function. (a) The function $p_{\rm true}$
	(b) The initial solution $p^{(0)}$ obtained by Step \ref{Step 3} in Algorithm \ref{alg}.
	(c) The function $p^{(5)}$ obtained by Step \ref{Step 8} in Algorithm \ref{alg}.
	(d) The true (solid), the initial solution (dot) and computed source function (dash-dot) on the diagonal line in (c).
	(e) The curve $\|p^{(k)} - p^{(k - 1)}\|_{L^{\infty}(\Omega)},$ $k = 1, \dots, 5.$
	The noise level of the data in this test is $20\%$.
	}
	\end{flushleft}
\end{figure} 

In this test, we successfully recover all four inclusions. 
On the other hand, the value of $p$ in each inclusion is high, making the true solution far away from the constant background $p_0 = 0$. 
Hence, $p_0 = 0$ might not serve as the initial guess.  
Our method to find the initial solution in Step \ref{Step 3} in Algorithm \ref{alg} is somewhat effective, see Figure \ref{m2 init20}. 
The computed images of the initial solution do not completely separate the inclusions.
Both computed values and images of the inclusions improve with iterations. 
The computed source function $p_{\rm comp} = p^{(5)}$ is acceptable, see Figure \ref{m2 p comp}. 
Figure \ref{m2 cross} shows that the constructed values in the inclusions are good. The procedure converges very fast, see Figure \ref{fig m2 error}.

The true maximal value of the upper left inclusion is 9 and the computed one is 8.992 (relative error 0.0\%).
The true maximal value of the upper right inclusion is 12 and the computed one is 13.4 (relative error 11.67\%).
The true maximal value of the lower left inclusion is 10 and the computed one is 10.13 (relative error 1.3\%).
The true maximal value of the lower right inclusion is 14 and the computed one is 14.86 (relative error 6.14\%).

\noindent {\bf Test 3.}
%In this test, we detect a ring. 
The true source function is given by
\[
	p_{\rm true} = \left\{
		\begin{array}{ll}
			1 & 0.2^2 < x^2 + y^2 < 0.8^2,\\
			0 & \mbox{otherwise.}
		\end{array}
	\right.
\]
The nonlinearity is given by
\[
	q(s) = s^2 \quad s \in \R.
\]
The support of the function $p_{\rm true}$ is ring-like.
This test is interesting due to the presence of the void and the nonlinearity grows fast.
The true and computed source functions $p$ are displayed in Figure \ref{example 3}.
\begin{figure}[h!]
	\begin{flushleft}
		\subfloat[]{\includegraphics[width = 0.3\textwidth]{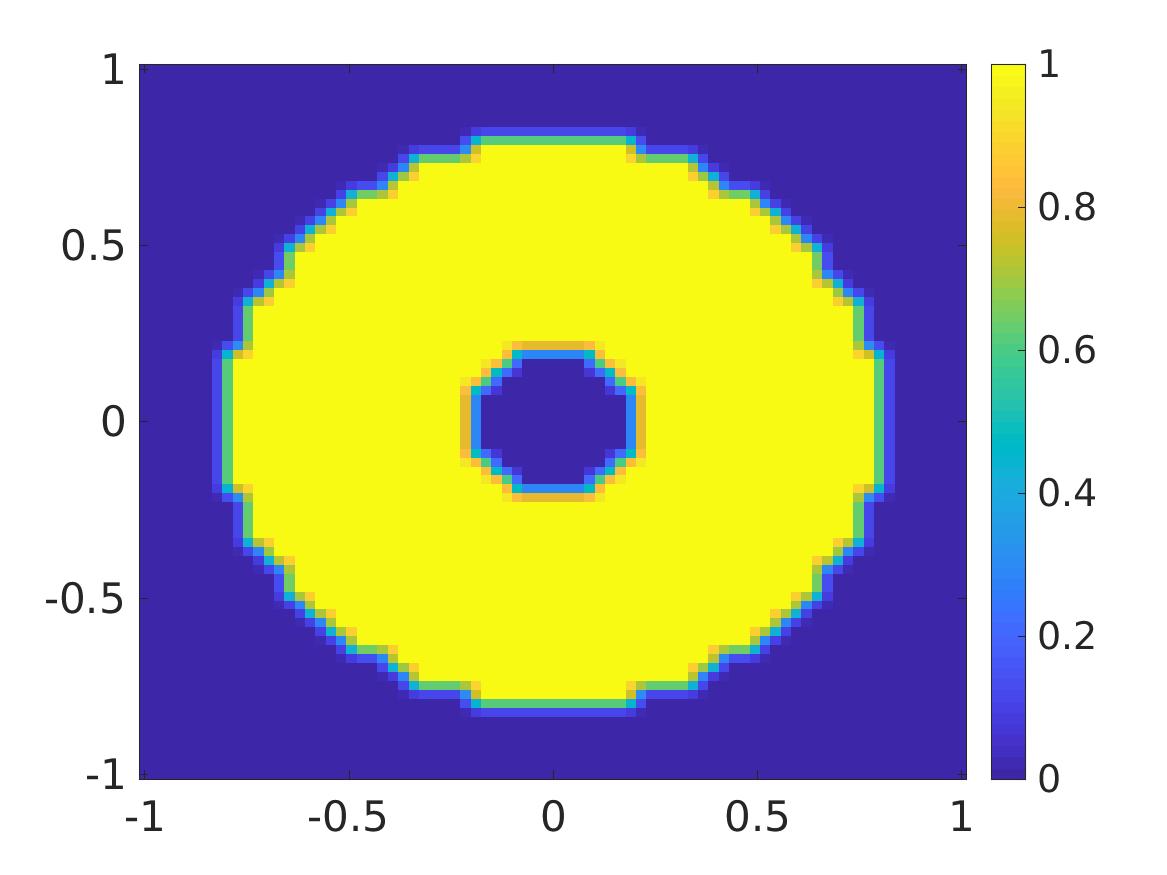}}
		\quad
		\subfloat[\label{m3 init20}]{\includegraphics[width = 0.3\textwidth]{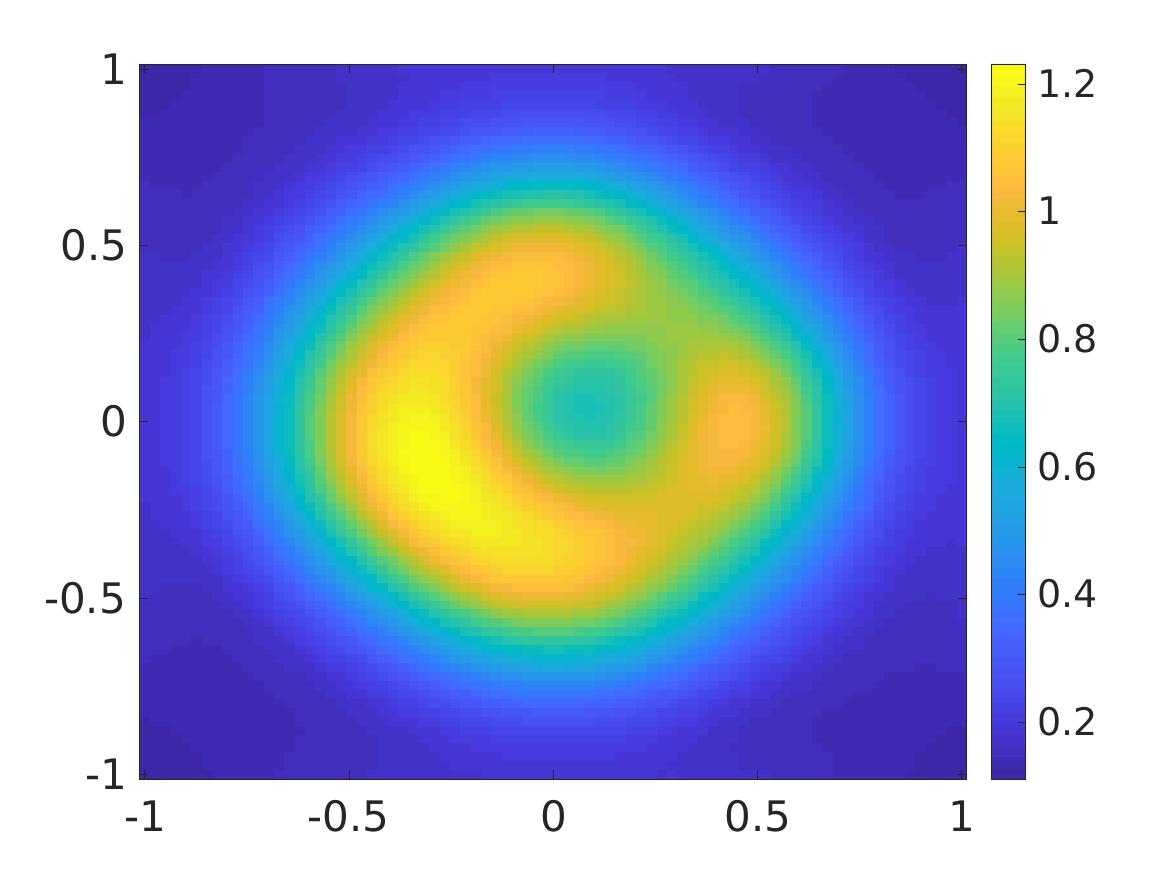}}  \quad

		\subfloat[\label{m3 p comp}]{\includegraphics[width = 0.3\textwidth]{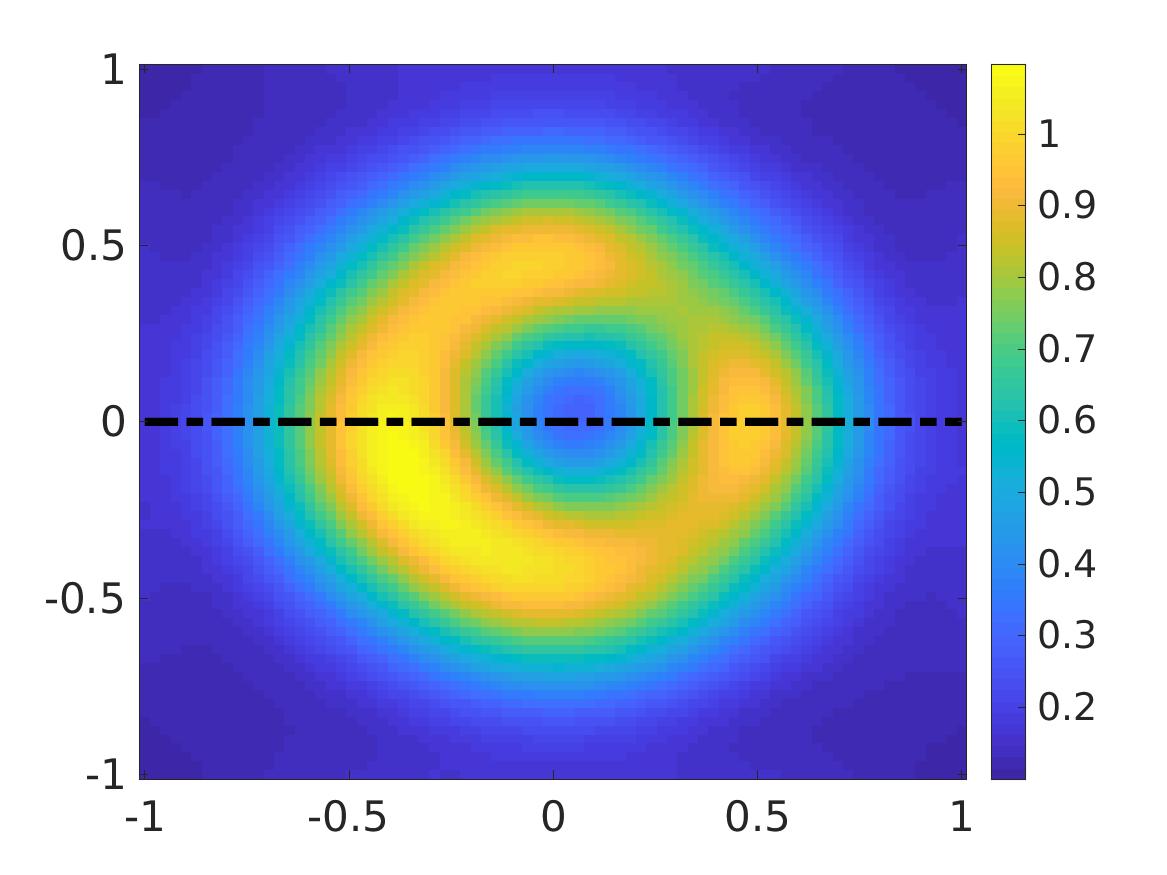}}\quad		
		\subfloat[\label{m3 cross}]{\includegraphics[width = 0.3\textwidth]{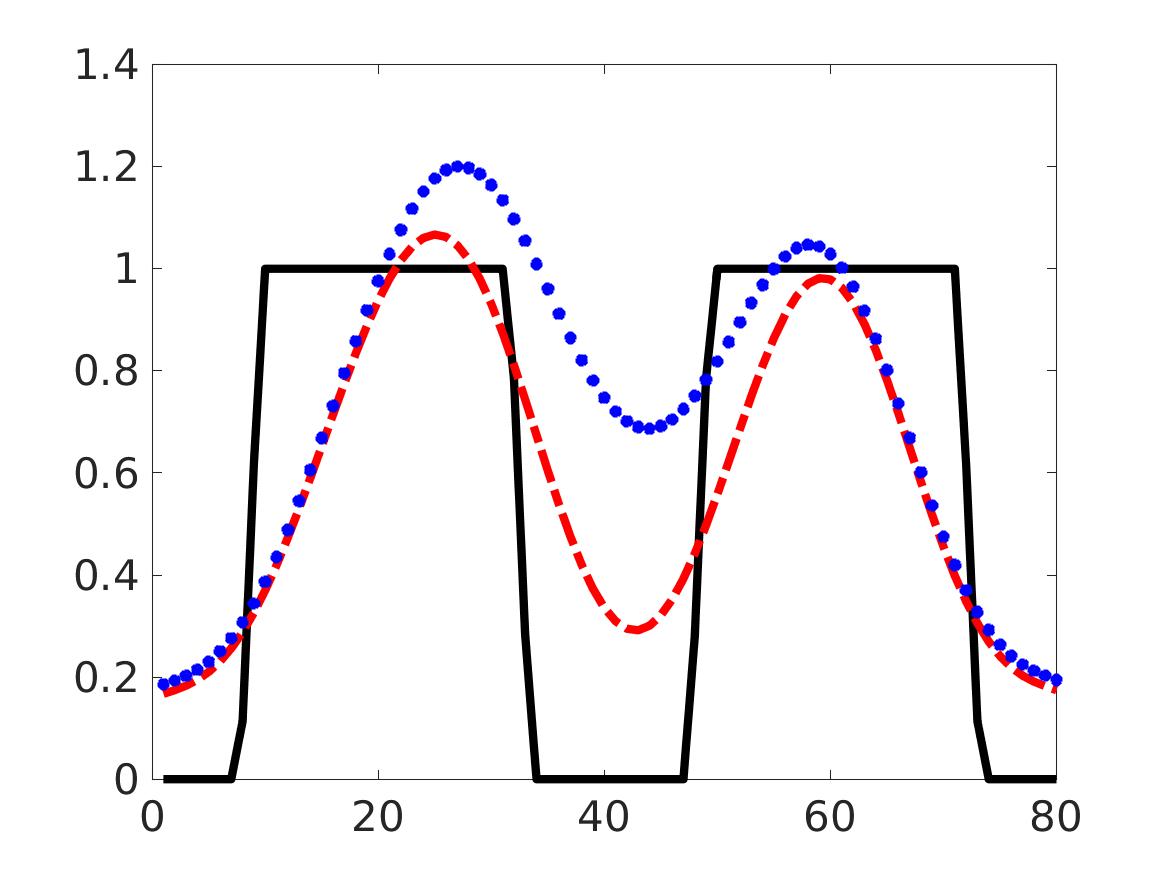}}
		\quad
		\subfloat[\label{fig m3 error}]{\includegraphics[width = 0.3\textwidth]{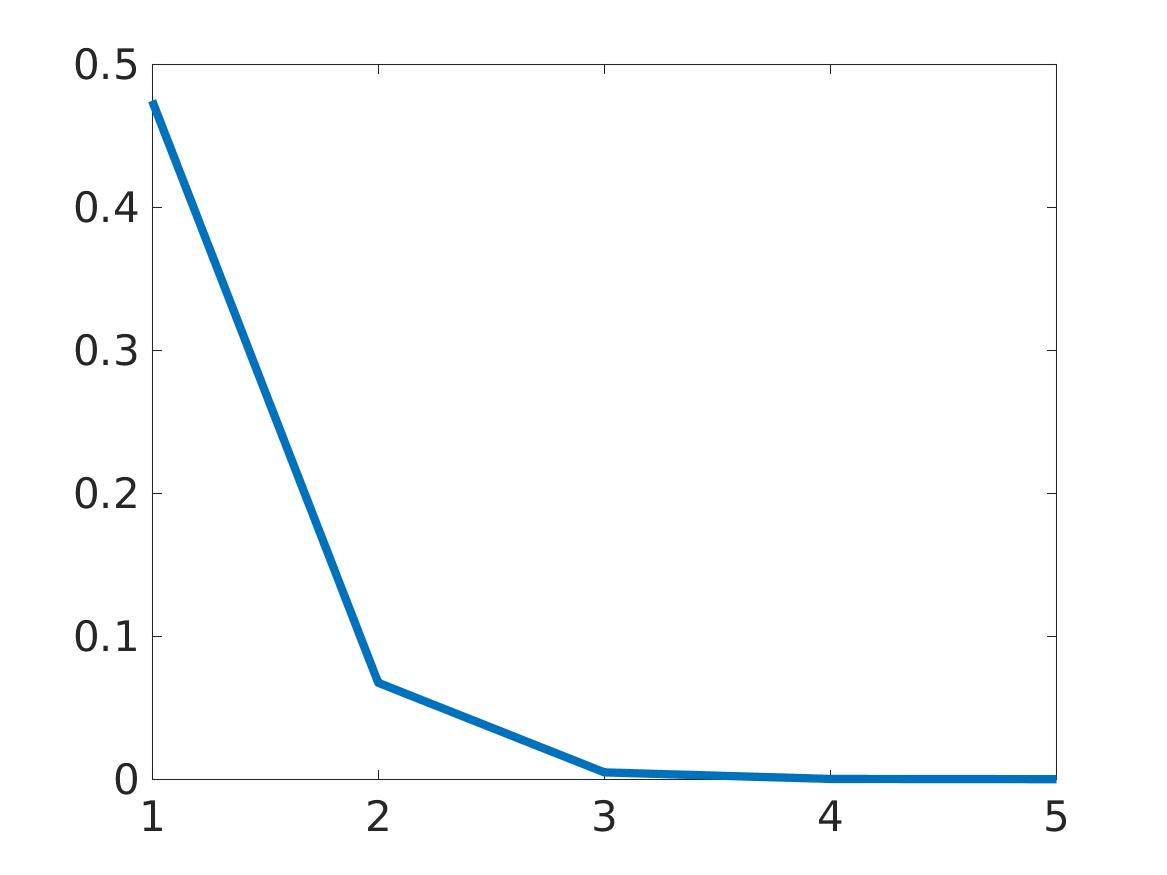}}
	\caption{\label{example 3} Test 3. The reconstruction of the source function. (a) The function $p_{\rm true}$
	(b) The initial solution $p^{(0)}$ obtained by Step \ref{Step 3} in Algorithm \ref{alg}.
	(c) The function $p^{(5)}$ obtained by Step \ref{Step 8} in Algorithm \ref{alg}.
	(d) The true (solid), initial solution (dot) and computed source function (dash-dot) on horizontal line  in (c).
	(e) The curve $\|p^{(k)} - p^{(k - 1)}\|_{L^{\infty}(\Omega)},$ $k = 1, \dots, 5.$ The noise level of the data in this test is $20\%$.
	}
	\end{flushleft}
\end{figure}

In this test, our method to find the initial solution in Step \ref{Step 3} in Algorithm \ref{alg} is somewhat acceptable. 
The void in the initial solution $p^{(0)}$ cannot be seen very well, see Figure \ref{m3 init20}.
The contrast and the void are improved with iteration. 
The final reconstructed source function $p^{(5)}$ is satisfactory, see Figures \ref{m3 p comp} and \ref{m3 cross}. 
The computed maximal value inside the ring is 1.094 (relative error = 9.4\%).

\noindent{\bf Test 4.} In this test, we identify two high contrast ``lines".
The true source function is given by
\[
	p_{\rm true} = \left\{
		\begin{array}{ll}
			10 & \max\{|x|/4, 4|y - 0.6| < 0.9\} \mbox{and } |x| < 0.8,\\
			8 &\max\{|x|/4, 4|y + 0.6| < 0.9\} \mbox{and } |x| < 0.8,\\
			0 &\mbox{otherwise.}
		\end{array}
	\right.
\]
The nonlinearity is given by
\[
	q(s) = - s^2 \quad s \in \R.
\]
The true and computed source functions $p$ are displayed in Figure \ref{example 4}.

\begin{figure}[h!]
	\begin{flushleft}
		\subfloat[]{\includegraphics[width = 0.3\textwidth]{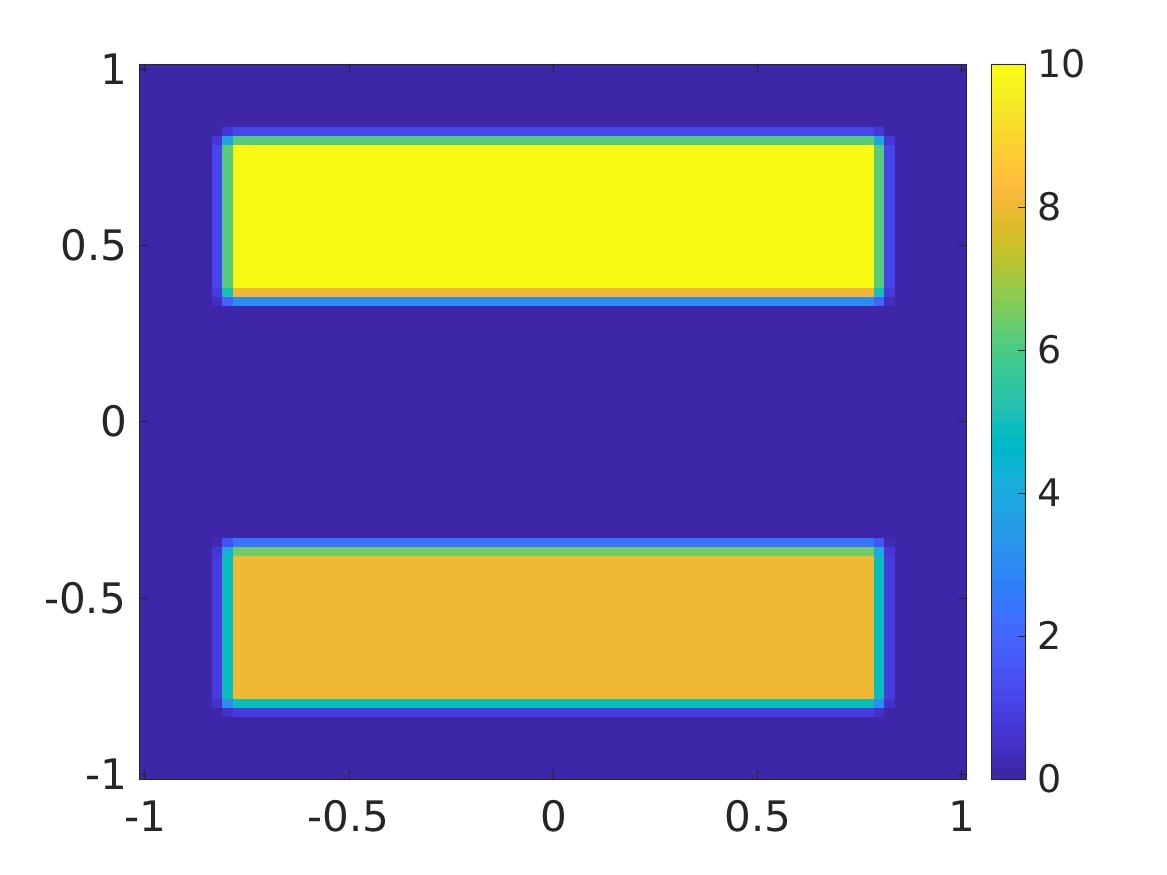}}
		\quad
		\subfloat[\label{m4 init20}]{\includegraphics[width = 0.3\textwidth]{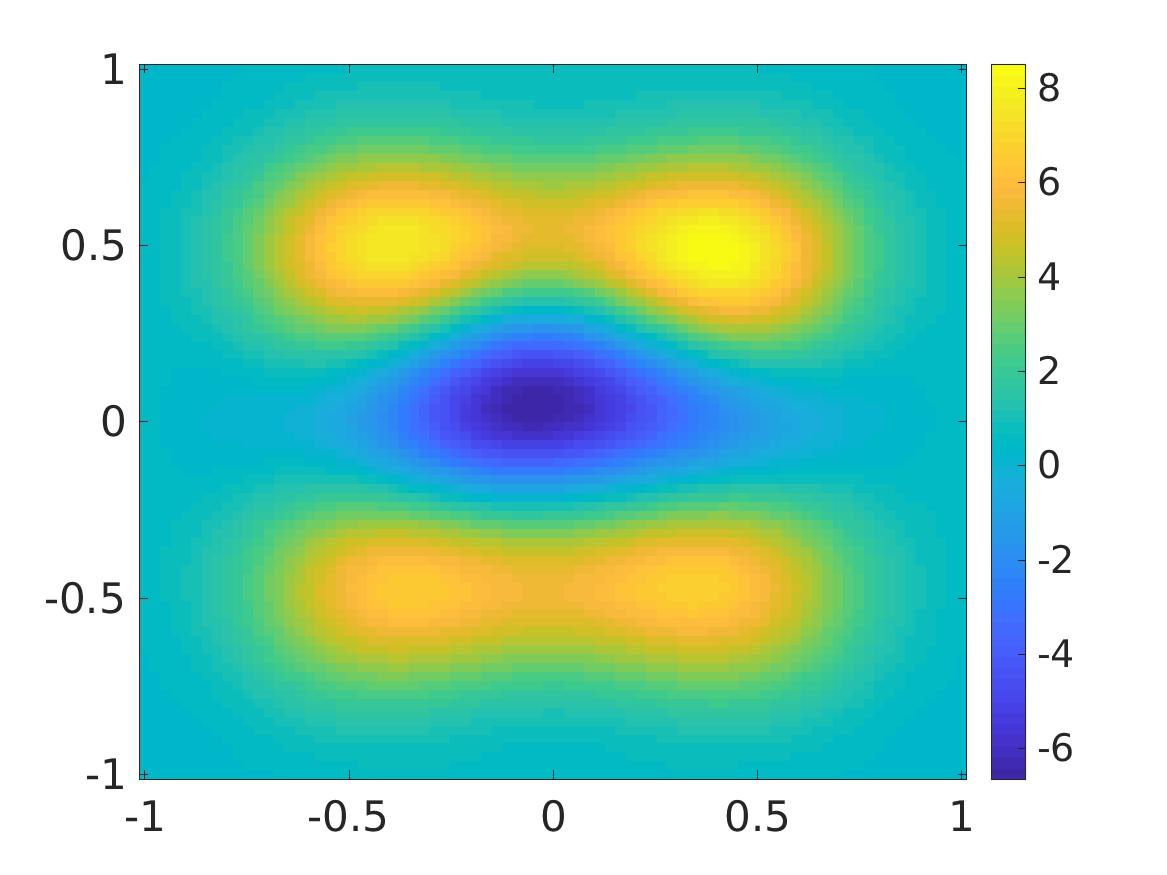}}  \quad

		\subfloat[\label{m4 p comp}]{\includegraphics[width = 0.3\textwidth]{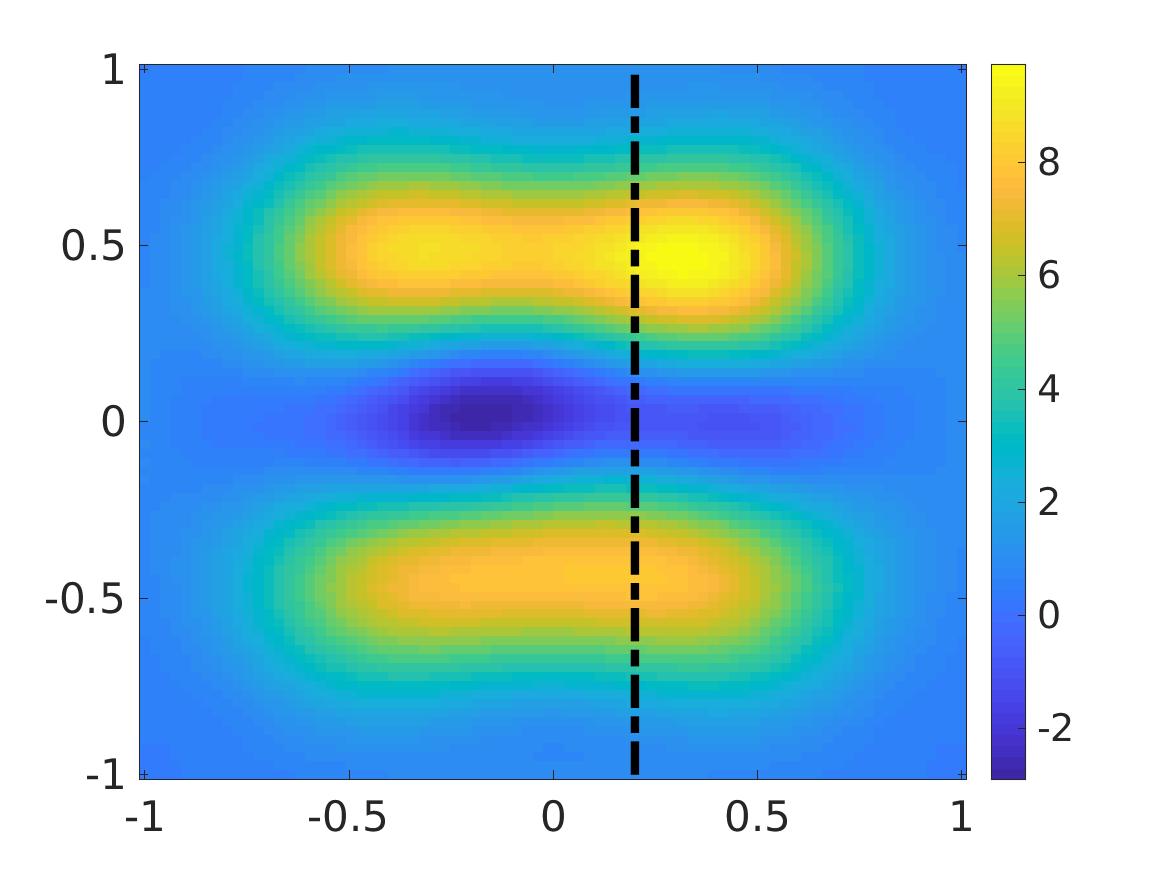}}\quad		
		\subfloat[\label{m4 cross}]{\includegraphics[width = 0.3\textwidth]{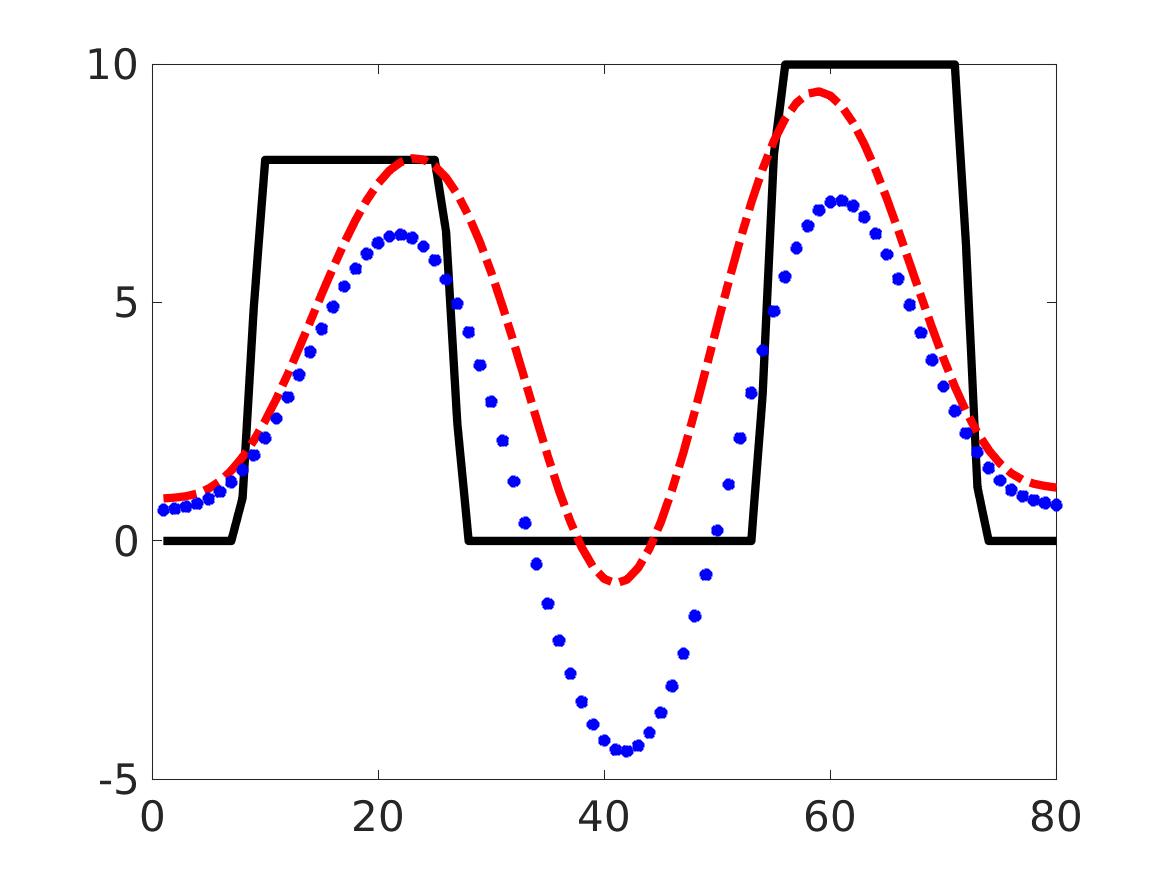}}
		\quad
		\subfloat[\label{fig m4 error}]{\includegraphics[width = 0.3\textwidth]{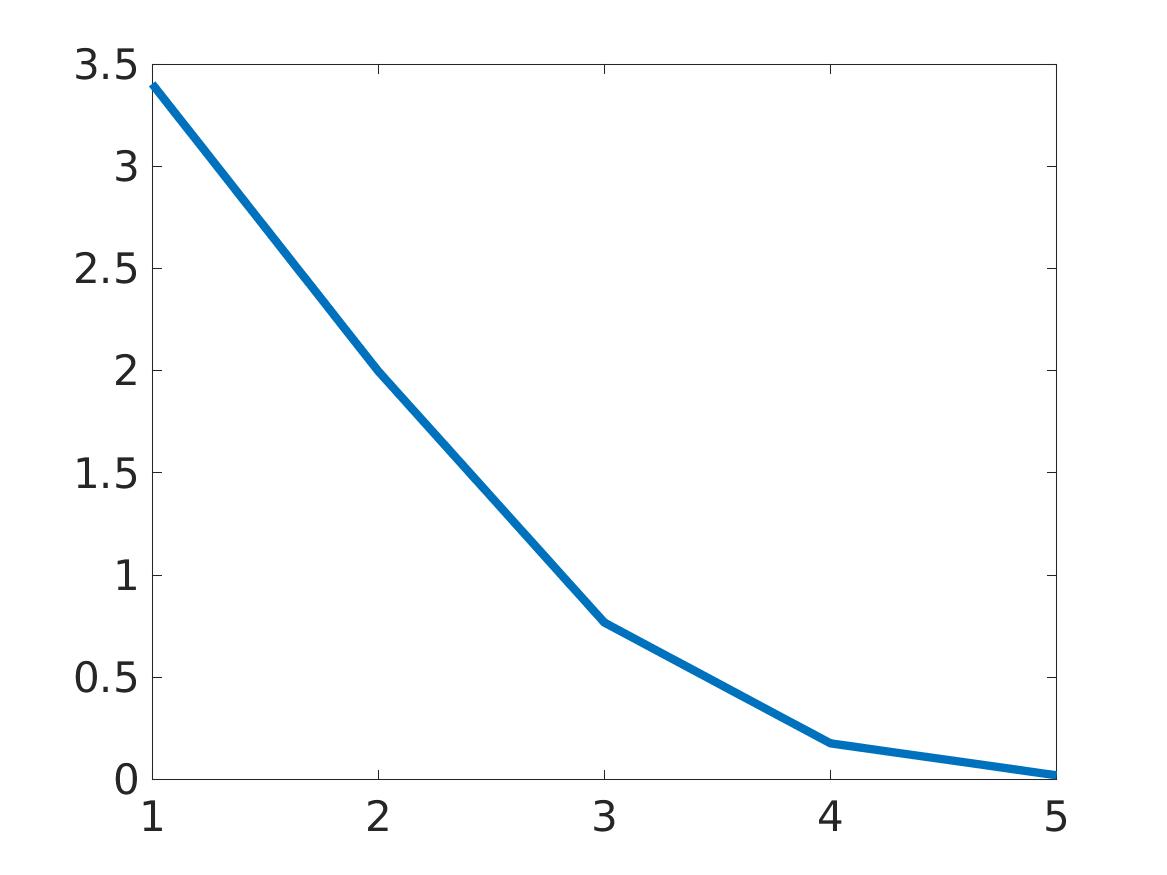}}
	\caption{\label{example 4} Test 3. The reconstruction of the source function. (a) The function $p_{\rm true}$
	(b) The initial solution $p^{(0)}$ obtained by Step \ref{Step 3} in Algorithm \ref{alg}.
	(c) The function $p^{(5)}$ obtained by Step \ref{Step 8} in Algorithm \ref{alg}.
	(d) The true and computed source function on the line (dash-dot) in (c).
	(e) The curve $\|p^{(k)} - p^{(k - 1)}\|_{L^{\infty}(\Omega)},$ $k = 1, \dots, 5.$
	The noise level of the data in this test is $20\%$.
	}
	\end{flushleft}
\end{figure} 

It is evident that Algorithm \ref{alg} provides a good computed source function. 
The initial solution by Step \ref{Step 3} in Algorithm \ref{alg} is quite good although there is a ``negative" artifact between the two detected lines, see Figure \ref{m4 init20}. This artifact is reduced significantly with iteration. 
We observe that the shape and contrasts of two lines are reconstructed very well, see Figures \ref{m4 p comp} and \ref{m4 cross}. 
Our method converges fast, see Figure \ref{fig m4 error}. 

The true maximal value of the source function in the upper line is 10 and the computed one is 9.714 (relative error 2.8\%).
The true maximal value of the source function in the lower line is 8 and the computed one is 8.041 (relative error 0.51\%).

\section{Concluding remarks} \label{sec remarks}

In this paper, we analytically and numerically solve the problem of recovering the initial condition of nonlinear parabolic equations. 
The first step in our method is to derive a system of nonlinear elliptic PDEs whose solutions are the Fourier coefficients of the solution to the governing nonlinear parabolic equation. 
We propose an iterative scheme to solve the system above.
Finding the initial solution for this iterative process is a part of our algorithm.
The convergence of this iterative method was proved. 
We show several numerical results to confirm the theoretical part.

\noindent{\bf Acknowledgment:}
The authors sincerely appreciate Michael V.Klibanov for many fruitful discussions that strongly improve the mathematical results and the presentation of this paper.
The work of the second author was supported by US Army Research Laboratory and US Army Research
Office grant W911NF-19-1-0044.

\providecommand{\bysame}{\leavevmode\hbox to3em{\hrulefill}\thinspace}
\providecommand{\MR}{\relax\ifhmode\unskip\space\fi MR }
\providecommand{\MRhref}[2]{%
  \href{http://www.ams.org/mathscinet-getitem?mr=#1}{#2}
}
\providecommand{\href}[2]{#2}

%\bibliography{../../../../../../mybib}{}
%\bibliographystyle{plain}

\end{document}